%% file: QI.tex
[1994/12/01]% LaTeX date must December 1994 or later
\documentclass[manuscript, printscheme]{aomart}

\fancypagestyle{firstpage}
\chardef\bslash=`\\ % p. 424, TeXbook
\newtheorem[{}\it]{thm}{Theorem}[section]
\newtheorem{cor}[thm]{Corollary}
\newtheorem{lem}[thm]{Lemma}
\newtheorem{prop}[thm]{Proposition}

\newtheorem{qus}[thm]{Question}
\newtheorem{rmk}[thm]{Remark}
\newtheorem{conj}[thm]{Conjecture}
\newtheorem{mainthm}{Theorem}

\theoremstyle{definition}
\newtheorem{defn}{Definition}[section]

\newtheorem*[{}\it]{notation}{Notation}

\newcommand{\thmref}[1]{Theorem~\ref{#1}}

\newcommand{\lemref}[1]{Lemma~\ref{#1}}
\newcommand{\corref}[1]{Corollary~\ref{#1}}
\newcommand{\propref}[1]{Proposition~\ref{#1}}

\newcommand{\defref}[1]{Definition~\ref{#1}}
\newcommand{\sref}[1]{Section~\ref{#1}}

\newcommand{\mA}{\mathcal{A}}

\newcommand{\mF}{\mathcal{F}}
\newcommand{\mU}{\mathcal{U}}
\newcommand{\N}{\mathbb{N}}
\newcommand{\Z}{\mathbb{Z}}
\newcommand{\T}{\mathbb{T}}
\newcommand{\R}{\mathbb{R}}

\newcommand{\wt}{\widetilde}

\newcommand{\eval}[2][\right]{\relax
  \ifx#1\right\relax \left.\fi#2#1\rvert}

\title[Quasi-isometric center]{Partially Hyperbolic Dynamics with Quasi-isometric Center}
\author[Feng]{Ziqiang Feng}
\address{BICMR\\
Peking University\\
Beijing, 100871, China}
\email{zqfeng@pku.edu.cn} 
\copyrightnote{}

\thanks{\formatdate{2024-11-16}}

\begin{document}

\begin{abstract}
  We consider the class of partially hyperbolic diffeomorphisms on a closed 3-manifold with quasi-isometric center. Under the non-wandering condition, we prove that the diffeomorphisms are accessible if there is no $su$-torus. As a consequence, volume-preserving diffeomorphisms in this context are ergodic in the absence of $su$-tori, thereby confirming the Hertz-Hertz-Ures Ergodicity Conjecture for this class. 
  
  We show the existence of transitive Anosov flows on a closed 3-manifold admitting a non-wandering partially hyperbolic diffeomorphism with quasi-isometric center and fundamental group of exponential growth. Furthermore, we provide a complete classification of these diffeomorphisms, showing they fall into two categories: skew products and discretized Anosov flows.
\end{abstract}

\maketitle
\tableofcontents

\input{introduction}

\input{model}
\input{complete}

\input{homoclinic}

\input{classify_transitive}
\input{accessible}
\input{classify_NW}

\section*{Acknowledgements}
I am grateful to Santiago Martinchich for updating me on the progress of their independent work and for his suggestion to present our work simultaneously. I also extend my sincere thanks to Ra\'{u}l Ures and Rafael Potrie for their reading and valuable comments on the manuscript.

\bibliographystyle{aomalpha}
\bibliography{QI_ref}

\end{document}

%% file: introduction.tex
% !Mode:: "TeX:UTF-8"

\section{Introduction}

The study of partially hyperbolic dynamics, regarded as an extension of uniformly hyperbolic dynamics, originated from the work of Brin and Pesin \cite{BrinPesin74} in the study of skew products and frame flows, and independently from Pugh and Shub \cite{PughShub72} in the study of Anosov actions. A partially hyperbolic diffeomorphism $f: M\rightarrow M$ of a closed Riemannian manifold $M$ presents hyperbolicity in its extreme subbundles, $E^s$ and $E^u$, within an invariant splitting $TM=E^s\oplus E^c\oplus E^u$. Specifically, the differential map $Df$ exponentially contracts vectors in $E^s$ and exponentially expands vectors in $E^u$, while the intermediate subbundle $E^c$ exhibits dominated pointwise spectrum. This center behavior results in intricate and subtlety in the study of partially hyperbolic dynamics.

Ergodicity and transitivity are two fundamental objects in partially hyperbolic dynamics from the measure-theoretical and topological perspectives, respectively. In contrast, partial hyperbolicity arises naturally in the theory of stable ergodicity and robust transitivity \cite{DPU}. 

Unlike uniformly hyperbolic systems, partially hyperbolic systems can formally lose stability in the topological conjugacy point of view under a small perturbation, even in dimension 3--the smallest dimension allowing for partial hyperbolicity. 
This sensitivity, coupled with the complex dynamical behavior in the center direction, complicates the determination of ergodic property for partially hyperbolic systems and whether we can effectively describe a partially hyperbolic diffeomorphism through a comparable model.

This paper addresses several problems related to ergodicity and transitivity within a broad class of partially hyperbolic diffeomorphisms. Our work engages with the Hertz-Hertz-Ures Ergodicity Conjecture, extending its verification within a class of diffeomorphisms displaying non-hyperbolic center behavior. In the topological context, we provide a sufficient condition for a closed 3-manifold to support a transitive Anosov flow, and present a complete classification for a somehow broad class of partially hyperbolic diffeomorphisms. Furthermore, we reveal an intrinsic correlation between ergodicity and transitivity. These seemingly unrelated problems are intrinsically intertwined in the current paper.

\subsection{Accessibility and ergodicity in dimension three}

The first result concerns the problem of determining dynamical or topological obstructions for a partially hyperbolic diffeomorphism to be ergodic in three-dimensional manifolds.

Building upon foundational techniques introduced by Hopf \cite{Hopf1939}, Anosov and Sinai \cite{Anosov1967, AnosovSinai67} extend the argument to establish the ergodicity of uniformly hyperbolic systems. However, the adaptation of Hopf argument to the partially hyperbolic context becomes significantly complicated by the presence of non-hyperbolic distribution. In response, the accessible property, capturing a form of pathwise hyperbolicity, serves as a decisive factor for ergodicity inspired by Pugh-Shub Stable Ergodicity Conjecture. Their conjecture also suggests that accessibility is a generic property for partially hyperbolic systems.  Subsequent developments, as demonstrated in \cite{08invent, BW10annals, ACW21}, confirm that ergodicity is an abundant property among volume-preserving partially hyperbolic diffeomorphisms in dimension three. Nonetheless, the existence of numerous simple non-ergodic examples underscores the inherent subtlety of the problem and leads to tremendous interest in inquiring how abundant ergodic diffeomorphisms are.

Following their presentation of the first manifold on which all conservative partially hyperbolic diffeomorphisms are ergodic \cite{2008nil}, Hertz, Hertz, and Ures formulated the Ergodicity Conjecture proposing a necessary topological obstruction to ergodicity in three-dimensional manifolds \cite{2011TORI,2018survey}.

\begin{conj}[Hertz-Hertz-Ures Ergodicity Conjecture]\label{HHU-conj}
	If a $C^r, r>1,$ conservative partially hyperbolic diffeomorphism of a closed 3-manifold does not admit any embedded 2-torus tangent to $E^s\oplus E^u$, then it is ergodic.
\end{conj}

The existence of such a 2-torus, as described in the conjecture above, imposes constraints on the topology of ambient manifolds, restricting them to one of three types of manifolds with virtually solvable fundamental groups \cite{2011TORI}.

A great deal of recent works has been dedicated to verifying the Hertz-Hertz-Ures Ergodicity Conjecture in specific manifolds \cite{2008nil,HamU14CCM,GS20DA,2020Seifert,FP_hyperbolic}. Beyond results on particular manifolds, notable progress has been made in affirming this conjecture for broader classes of partially hyperbolic diffeomorphisms \cite{FP_hyperbolic,FP21accessible,FU1}. It is worth mentioning that \cite{FU1} establishes the ergodicity of all conservative partially hyperbolic diffeomorphisms without periodic points regardless of the structure of underlying closed 3-manifold. This result contrasts with prior works, which either focus on specific manifolds or rely on assumptions about the topology of the ambient manifold.

In this work, we provide a positive answer to Hertz-Hertz-Ures Ergodicity Conjecture for a class of partially hyperbolic diffeomorphisms exhibiting non-hyperbolic center behavior.

We denote by $NW(f)$ the non-wandering set of a partially hyperbolic diffeomorphism $f$ throughout the paper. We say that a partially hyperbolic diffeomorphism $f: M\rightarrow M$ has \emph{quasi-isometric center} if there are constants $r, R>0$ and an $f$-invariant center foliation $\mF^c$ such that $f^n(\mF^c_r(x))\subset \mF^c_R(f^n(x))$ for any $x\in M$ and any $n\in \Z$. Equivalent formulations and a slightly refined definition of quasi-isometric action in the center are discussed in Section \ref{section-pre} and Section \ref{section-complete}. 

It is noteworthy that this class of partially hyperbolic diffeomorphisms has also been studied in \cite{CP22_principle}, which sheds light on certain properties of ergodic measures.

\begin{mainthm}\label{accessible}
	Let $f: M^3\rightarrow M^3$ be a $C^1$ partially hyperbolic diffeomorphism with quasi-isometric center on a closed 3-manifold whose fundamental group is not virtually solvable. If $NW(f)=M$, then $f$ is accessible. In particular, if $f$ is $C^r$ with $r>1$ and conservative, then it is a (stably) K-system and thus (stably) ergodic.
\end{mainthm}

Recently, Meysam Nassiri proposed us a long-considered conjecture that establishes a connection between two fundamental properties in partially hyperbolic dynamics: ergodicity and transitivity. This conjecture, which can be deduced from Conjecture \ref{HHU-conj} when the manifold has non-virtually solvable fundamental group, is particularly anticipated to hold in the case of one-dimensional center. We formulate this conjecture here in the context of three-dimensional manifolds:
\begin{conj}[Nassiri]
	Every $C^{1+\alpha}$ transitive volume-preserving partially hyperbolic diffeomorphism of a closed 3-manifold is ergodic.
\end{conj}

Our next result builds on the conjecture above, demonstracting that under the quasi-isometric center condition, transitivity and ergodicity are indeed fundamentally aligned.

\begin{mainthm}\label{ergodic}
	Let $f: M^3\rightarrow M^3$ be a $C^2$ conservative partially hyperbolic diffeomorphism of a closed 3-manifold with quasi-isometric center. Then, it is ergodic if and only if it is transitive.
\end{mainthm}

We emphasize that the theorem above is sharp within the given context, as it would not hold if the transitivity condition is weakened to the non-wandering property. Indeed, one can readily contruct non-wandering examples that fail to be ergodic; for instance, the direct pruduct of an Anosov diffeomorphism with the identity map on the 3-torus serves as a clear counterexample.

\subsection{Existence of Anosov flows}

Our next result highlights a relevance between the existence of partially hyperbolic diffeomorphisms and Anosov flows on 3-manifolds.

The study of Anosov flows originated as a natural generalization of geodesic flows on negatively curved closed Riemannian manifolds. The foundational works of Margulis \cite{Margulis67} and Plante-Thurston \cite{PT72} established that exponential growth of the fundamental group is a necessary condition for a 3-manifold to admit an Anosov flow. Classic examples of transitive Anosov flows include suspensions of hyperbolic toral automorphisms and geodesic flows on unit tangent bundles of negatively curved manifolds. Beyond these, the landscape of Anosov flows is significantly richer, with notable constructions by Franks–Williams \cite{FranksWilliams}, Handel–Thurston \cite{HT80}, Goodman \cite{Goodman83}, and Fried \cite{Fried83}. A pivotal development in the field has been the use of topological surgeries, particularly Dehn–Fried–Goodman surgery, as a technique for constructing Anosov flows. 

Due to diversity of such surgeries, the topological classification of Anosov flows in 3-manifolds becomes challenging. Complete classifications have been achieved for solvable manifolds \cite{Plante81}, Seifert manifolds \cite{Ghys84}, and certain classes of graph manifolds \cite{Barbot98}. Building on the fundamental contributions of Thurston and Perelman, hyperbolic 3-manifolds have garnered significant attention regarding obstructions to the existence of Anosov flows. Notably, there are infinitely many hyperbolic 3-manifolds that admit Anosov flows, as shown by Goodman \cite{Goodman83} and Thurston \cite{Thurston78-note}, while there are also infinitely many hyperbolic 3-manifolds that do not, due to the absence of Reebless foliations \cite{CalegariDunfield03,RSS03} or essential laminations \cite{Fenley07}. Recent progress on classifications of Anosov flows can be found in \cite{BarbotFenley13, Yu23, BFM22transitive, BFM23, BBM24nontransitive}.

The interplay between Anosov flows and the topology of 3-manifolds is profound. Similarly, the existence of partially hyperbolic diffeomorphisms also influences 3-manifold topology \cite{BI08,2011TORI}. Investigating the relationship between the existence of these two dynamical systems remains an active area of interest. One can notice that, while every Anosov flow generates a partially hyperbolic diffeomorphism via its time-one map, the converse is not always true; 3-tori and nilmanifolds, for example, admit partially hyperbolic diffeomorphisms but not Anosov flows.

This paper addresses topological obstructions to the existence of Anosov flows arising from partially hyperbolic systems. Specifically, we examine the following question:

\begin{qus}
	Let $M$ be a closed 3-manifold whose fundamental group has exponential growth and which admits a (transitive) partially hyperbolic diffeomorphism. Does $M$ support a (transitive) Anosov flow?
\end{qus}

This problem was recently studied in \cite{FP-gafa} for hyperbolic 3-manifolds and certain class of Seifert manifolds. We provide a partial answer to this question for a broad class of partially hyperbolic diffeomorphisms.
\begin{mainthm}\label{flow}
	Let $M$ be a closed 3-manifold admitting a non-wandering partially hyperbolic diffeomorphism with quasi-isometric center. If the fundamental group $\pi_1(M)$ has exponential growth, then $M$ admits a transitive Anosov flow.
\end{mainthm}

This result represents a significant advance over previous classifications, as it imposes no assumptions on the specific topology of the ambient 3-manifold beyond the necessary condition of exponential growth for $\pi_1(M)$. Moreover, it addresses a notable challenge: finite covers, which often complicate the existence of Anosov flows \cite{CalegariDunfield03, RSS03}, are not required here, highlighting the generality of this result.

\subsection{Classification}
In the study of partially hyperbolic dynamics, a common approach has been to infer dynamical properties by analyzing explicit well-understood examples. Classical examples of 3-dimensional partially hyperbolic diffeomorphisms include systems derived from Anosov dynamics, skew products over Anosov automorphisms of the 2-torus, and variable time maps of Anosov flows. As previously noted, these systems are highly sensitive to perturbation at the level of topological conjugacy, adding significant delicacy to the problem of classifying partially hyperbolic diffeomorphisms topologically. 

This classification problem can be traced back to an informal conjecture by Pujals, later formulated by Bonatti and Wilkinson in \cite{BW05}. Roughly, the conjecture suggests that the three aforementioned classes form the foundation of a comprehensive list of transitive 3-dimensional partially hyperbolic diffeomorphisms. 

Significant progress towards classifying 3-dimensional partially hyperbolic diffeomorphisms has been achieved through two primary approaches. The first involves comparing specific classes of diffeomorphisms, either within certain manifolds or exhibiting distinct dynamical properties, to known models. This method culminated in the confirmation of the Pujals conjecture by Hammerlindl and Potrie \cite{Hammerlindl13, HP14, HP15} in case where the ambient closed 3-manifold has virtually solvable fundamental group. 

The second approach focuses on analyzing the local and global structures of dynamical foliations and their impact on the ambient manifolds. Locally, Bonatti and Wilkinson \cite{BW05} provided criteria that classify diffeomorphisms into the categories of skew products and discretized Anosov flows. These two classes comprehensively account for diffeomorphism exhibiting topologically neutral behavior in the center direction \cite{BZ20}. Globally, a desired classification has been achieved across a wide range of 3-manifolds \cite{BFFP1, BFFP2}. Furthermore, under specific rigid conditions, partially hyperbolic diffeomorphisms can also be conclusively categorized into the three fundamental classes outlined earlier \cite{CPH21}.

The classification problem has become increasingly challenging with the emergence of several anomalous examples \cite{BPP1,BGP2,BGHP3}. To accommodate these examples, recent works \cite{BFP23collapsed,FP23_transverse} introduced the concept of \emph{collapsed Anosov flows}, which broadens the framework for classifying partially hyperbolic diffeomorphisms and offers fresh insights into the problem. We refer readers to the surveys \cite{2018survey,HP18_survey,Potrie18ICM} for an account of this topic.

It is important to note that the classification problem is typically studied modulo finite lifts and iterates, as both topological and dynamical properties can change significantly under these operations. For instance, examples exist where the center foliation forms a non-trivial circle bundle, but becomes a skew products after a finite lift \cite{BW05}. Similarly, a lift of a discretized Anosov flow to a finite cover may fail to preserve all flow orbits. The existence of Anosov flows themselves is sensitive to finite lifts \cite{CalegariDunfield03,RSS03}, and the same sensitivity applies to partially hyperbolic diffeomorphisms conjectured to be discretized Anosov flows.

In this paper, we are going to provide a complete classification of the class of partially hyperbolic diffeomorphisms with quasi-isometric center, addressing the original problem posed by Pujals. Notably, the quasi-isometric center property is strictly weaker than the topologically neutral center condition studied in \cite{BZ20}. 

\begin{mainthm}\label{classify_NW_strong}
	Let $f: M^3\rightarrow M^3$ be a $C^1$ partially hyperbolic diffeomorphism of a closed 3-manifold with quasi-isometric center and $NW(f)=M$. Then, either
	\begin{itemize}
		\item up to a finite lift and iterate, $f$ is conjugate to a skew product over an Anosov automorphism on $\mathbb{T}^2$; or 
		\item $f$ admits an iterate that is a discretized Anosov flow.
	\end{itemize}
\end{mainthm}

A significant consequence of this classification is that the two classes of partially hyperbolic diffeomorphisms listed above provide a complete framework for applying the invariance principle established in \cite{CP22_principle} within the 3-dimensional setting.

During the preparation of this work, we became aware of a similar result obtained independently by Espitia, Martinchich and Potrie \cite{EMP}. Although both approaches address the same problem, we employ distinct techniques and perspectives. For this reason, we decide to present our results separately and simultaneously to provide complementary insights into the problem.

\subsection{Remarks and organization}
In the definition of a partially hyperbolic diffeomorphism with quasi-isometric center (see also \defref{QIdef-leaf}), we assume the existence of an invariant center foliation. However, as demonstrated by the examples in \cite{2016example,BGHP3}, such a foliation does not always exist. To address this, we propose an alternative notion that retains the quasi-isometric action in the center direction (see \defref{QIdef-curve}). As show in \propref{integrable}, integrability along the center still holds under the new definition. Furthermore, this alternative notion is, in fact, more restrictive than the original definition provided in \defref{QIdef-leaf} (see \propref{2QI}).

Although the results \thmref{accessible}, \thmref{ergodic}, \thmref{flow}, and \thmref{classify_NW_strong} are stated for the quasi-isometric center (foliation) as defined in \defref{QIdef-leaf}, the arguments can be directly adapted to partially hyperbolic diffeomorphisms with quasi-isometric center bundle, as defined in \defref{QIdef-curve}. For the sake of brevity and conciseness, we do not repeat these adaptations in this paper.

Another remark is that, in addition to \thmref{classify_NW_strong}, we present several weaker classification results, see \thmref{classify_transitive} and \thmref{classify_NW}. In \thmref{classify_transitive}, a classification is achieved under a stronger assumption of transitivity. \thmref{classify_NW} employs the same assumption as \thmref{classify_NW_strong} but includes an additional class, the collapsed Anosov flow, in its conclusion. These weaker results are presented due to their intrinsic connections: on the one hand, \thmref{accessible}, \thmref{classify_transitive}, and \thmref{classify_NW} can each be proved using independent arguments; on the other hand, \thmref{classify_NW} can be used to derive \thmref{accessible}, and conversely, it can be deduced from \thmref{classify_transitive} and \thmref{accessible}. Further details are provided in Section \ref{section-classify-NW}.

The structure of this paper is as follows. In Section \ref{section-pre}, we introduce key concepts related to the quasi-isometric center and two models central to this study. Section \ref{section-complete} discusses several definitions of central quasi-isometry and their relevance, as well as the integrability of the center bundle and the structure of center-stable and center-unstable foliations. We devote Section \ref{section-periodic} to the study of periodic compact center leaves. A classification under the transitivity assumption is presented in Section \ref{section-classify-transitive}, while Section \ref{section-accessible} contains the proofs of \thmref{accessible} and \thmref{ergodic}. Finally, Section \ref{section-classify-NW} provides a coarser classification for the non-wandering case and includes the proofs of \thmref{flow} and \thmref{classify_NW_strong}.

%% file: model.tex
% !Mode:: "TeX:UTF-8"

\section{The target models}\label{section-pre}

In this section, we clarify some notions relavent to partially hyperbolic dynamics that will be concerned with in this paper. We also provide a concise overview of two natural classes of partially hyperbolic diffeomorphisms with quasi-isometry center behavior. These two classes encompass all such systems with non-wandering property module finite lifts and iterates, as described in \thmref{classify_NW_strong}.

\subsection{Quasi-isometry in the center}
Let $M$ be a compact Riemannian manifold. We say a diffeomorphism $f: M\rightarrow M$ is \emph{(strongly pointwise)partially hyperbolic} if the tangent bundle of $M$ splits into three nontrivial invariant subbundles $TM=E^s \oplus E^c \oplus E^u$ such that for each $x\in M$ and for all unit vectors $v^\sigma \in E^\sigma_x$ ($\sigma= s, c, u$), 
\begin{equation*}
	\|Df(x)v^s\|<1<\|Df(x)v^u\| \\
	\quad and \quad
	\|Df(x)v^s\|< \|Df(x)v^c\|< \|Df(x)v^u\|.
\end{equation*}

The existence of stable and unstable foliations plays a critical role in the study of Anosov systems. The unique integrability of strong bundles for partially hyperbolic diffeomorphisms is established by the fundamental work \cite{BrinPesin74, PughShub72}. However, the integrability of center bundle is quite a delicate problem in the partially hyperbolic dynamics.

We say that $E^c$ is \emph{uniquely integrable} if there is an $f$-invariant integral foliation $\mF^c$ by $C^1$ leaves such that every $C^1$ embedded curve $\sigma: [0,1]\rightarrow M$ satisfying $\dot{\sigma}(t)\in E^c(\sigma(t))$ for any $t\in [0,1]$ is contained in the leaf $\mF^c(\sigma(0))$ through the initial point $\sigma(0)$. A partially hyperbolic diffeomorphism $f$ is \emph{dynamically coherent} if it admits two invariant foliations $\mF^{cs}$ and $\mF^{cu}$ tangent to the bundles $E^s\oplus E^c$ and $E^c\oplus E^u$, respectively. The dynamical coherence is a direct consequence of the unique integrability of $E^c$. We refer the reader to \cite{BurnsWilkinson08} for an account of relations between different types of central integrability.

\begin{defn}\label{QIdef-leaf}
	Let $f:M\rightarrow M$ be a $C^1$ partially hyperbolic diffeomorphism of a closed manifold with one-dimensional center bundle. We say that $f$ has \emph{quasi-isometric center foliation} if there exists an invariant center foliation $\mathcal{F}^c$ and constants $A\geq 1$ and $B>0$ such that for any curve $c$ inside a leaf of $\mathcal{F}^c$, we have that 
	$$
	A^{-1}\cdot L(c)-B\leq L(f^n(c))\leq A\cdot L(c)+B
	$$
	for any $n\in \N$, where $L(\cdot)$ denotes the length of a curve.
\end{defn}

The definition above is the usual one used in the literature. Compared with this definition, we propose an alternative way to define the central quasi-isometry where the collection of integral curves tangent to $E^c$ is not necessarily assembled to a foliation.

\begin{defn}\label{QIdef-curve}
	Let $f:M\rightarrow M$ be a $C^1$ partially hyperbolic diffeomorphism of a closed manifold with one-dimensional center bundle. We say that $f$ has \emph{quasi-isometric center bundle} if there exist constants $A\geq 1$ and $B>0$ such that for any $C^1$ curve $c$ tangent to $E^c$, we have that 
	$$
	A^{-1}\cdot L(c)-B\leq L(f^n(c))\leq A\cdot L(c)+B
	$$
	for any $n\in \N$, where $L(\cdot)$ denotes the length of a curve.
\end{defn}

We will discuss about the similarities and differences between these two notions of quasi-isometry in Section \ref{section-complete}, and moreover, present several equivalent statements. In the subsequent sections, we will refer the term "quasi-isometric center" as quasi-isometric center foliation defined by \defref{QIdef-leaf}.

\subsection{Skew product}

The skew product systems are a type of diffeomorphisms that fiber over a lower dimensional diffeomorphism. In the context of partially hyperbolic diffeomorphisms, the most natural way is to associate the center bundle to the tangent space of fibers. A toy example of skew product partially hyperbolic diffeomorphism is the map $(A, R_{\theta}): \T^2\times S^1\rightarrow \T^2\times S^1$ with tirvial product given by $(A, R_{\theta})(x, t)=(Ax, t+\theta)$, where $A$ is an Anosov automorphism and $R_{\theta}$ is a rotation. 

The usual skew product is defined in the following way. Let $N$ be a smooth manifold and $G$ be a compact Lie group. We say that a diffeomorphism $f: N\times G\rightarrow N\times G$ is a skew product if there exist diffeomorphisms $A: N\rightarrow N$ and $\theta: M\rightarrow G$ such that $f(x, g)=(Ax, \theta(x)g)$.

We will refer the skew products as a more general class of diffeomorphisms which does not necessarily require a trivial bundle or even a principle bundle. We say that a diffeomorphism $f: M\rightarrow M$ of a closed manifold is a \emph{skew product over a map} $A: N\rightarrow N$ if it preserves a fibration $\pi: M\rightarrow N$ and the projection map coincides with $A$. Note that the manifold $M$ and the bundle $\pi: M\rightarrow N$ might not be orientable. Furthermore, the restriction map on each fiber may not be an isometry. 

In dimension three, the manifold $N$ given in the skew product definition above is the torus $\T^2$ and the fibers are homeomorphic to $S^1$. One can easily construct a skew product partially hyperbolic diffeomorphism whose restriction map on the fibers is not an isometry by making a perturbation of the product map $A\times id: \T^2\times S^1\rightarrow \T^2\times S^1$ for an Anosov diffeomorphism $A$. 

A natural class of partially hyperbolic diffeomorphisms with quasi-isometric center of a closed 3-manifold is those with a uniformly compact center foliation. Indeed, every dynamically coherent skew product partially hyperbolic diffeomorphism in dimension three has a uniformly compact center foliation and thus has quasi-isometric center, see \propref{QI-skewproduct}.

\subsection{Discretized Anosov flow}

Let us first recall the concept of topological Anosov flows, which are topological generalizations of Anosov flows.

A \emph{topological Anosov flow} is a continuous flow $\phi_t: M\rightarrow M$ generated by a continuous non-singular vector field $X=\frac{\partial \phi_t}{\partial t}|_{t=0}$ satisfying that:
\begin{enumerate}
	\item The flow $\phi_t$ preserves two topologically transverse codimension-one continuous foliations $\mF^{ws}$ and $\mF^{wu}$;
	\item For any $x, y\in M$ in the same leaf of $\mF^{ws}$ (resp. $\mF^{wu}$), there exists a continuous increasing reparametrization $h: \R\rightarrow \R$ such that $d(\phi_t(x), \phi_{h(t)}(y))\rightarrow 0$ as $t\rightarrow +\infty$ (resp. $t\rightarrow -\infty$);
	\item There exists $\delta>0$ such that for any $x\in M$ and $y\in \mF^{ws}_{\delta}(x)$ (resp. $y\in \mF^{wu}_{\delta}(x)$), with $y$ not in the same orbit as $x$, and for any continuous increasing reparametrization $h: \R\rightarrow \R$, there exists $t\leq 0$ (resp. $t\geq 0$) such that $d(\phi_t(x), \phi_{h(t)}(y))>\delta$.
\end{enumerate}

Topological Anosov flows have been studied as non-singular pseudo-Anosov flows, see for instance \cite{Mosher92-1, Mosher92-2}, and a generalization of Anosov flows without dominated splitting, see also \cite{Shannon_flow}. 

We say that a non-singular flow $\phi_t: M\rightarrow M$ is \emph{expansive} if for every $\epsilon>0$ there exists $\delta>0$ such that for an increasing homeomorphism $h: \R\rightarrow \R$ with $h(0)=0$ and any points $x, y\in M$, if $d(\phi_t(x), \phi_{h(t)}(y))\leq \delta$ for all $t\in \R$, then $y=\phi_s(x)$ for some $|s|<\epsilon$. The expansiveness of a non-singular flow has the same nature as item (3) in the definition above of a topological Anosov flow. The purpose of constant $\epsilon$ is to prevent the self-accumulation of a flow orbit over a long time.

In light of previous work of Inaba and Matsumoto \cite{IM90}, and Paternain \cite{Paternain93}, the definition of topological Anosov flows can be given in an alternative way.

\begin{thm}\cite{IM90, Paternain93}\cite[Theorem 5.9]{BFP23collapsed}
	A non-singular continuous flow $\phi_t: M\rightarrow M$ is a topological Anosov flow if and only if it is expansive and preserves a foliation.
\end{thm}

We say that a partially hyperbolic diffeomorphism $f: M\rightarrow M$ of a closed manifold with one-dimensional center is a \emph{discretized Anosov flow} if there is a topological Anosov flow $\phi_t: M\rightarrow M$ and a continuous function $\tau: M\rightarrow \R$ such that $f(x)=\phi_{\tau(x)}(x)$ for any $x\in M$. Equivalently, a partially hyperbolic diffeomorphism $f: M\rightarrow M$ with one-dimensional center is a discretized Anosov flow if it admits an invariant center foliation $\mF^c$ such that $f(x)\in \mF^c_{L}(x)$ for any $x\in M$, where $L>0$ is a constant and $\mF^c_L(x)$ is the ball centered at $x$ of radius $L$ within the leaf $\mF^c(x)$. We refer the reader to \cite{BFFP1, BFP23collapsed} and \cite[Section 3.7]{Martinchich23} for more discussion on discretized Anosov flows and other equivalent definitions.

One can easily see that the time-one map of an Anosov flow is a discretized Anosov flow, as well as its perturbations. Note that discretized Anosov flow is a wider class than time-one maps of (topological) Anosov flows. This provides a class of partially hyperbolic diffeomorphisms with quasi-isometric center foliation \cite{Martinchich23}. But as we will see in Section \ref{section-complete}, not every discretized Anosov flow has quasi-isometric center bundle.

%% file: complete.tex
% !Mode:: "TeX:UTF-8"

\section{Central unique integrability}\label{section-complete}

In this section, we explore two notions of central quasi-isometry introduced in Section \ref{section-pre} for partially hyperbolic diffeomorphisms with one-dimensional center. The definition provided in \defref{QIdef-leaf} assumes the existence of an invariant center foliation, while \defref{QIdef-curve} releases this assumption. As we will see, although the second definition does not require dynamical coherence, the central quasi-isometric property nonetheless indicates the integrability of center bundle.

\subsection{Two concepts of central quasi-isometry}

We first provide some equivalent statements with two definitions given by \defref{QIdef-leaf} and \defref{QIdef-curve}.
\begin{lem}\label{QIcurve=}
	Let $f: M\rightarrow M$ be a partially hyperbolic diffeomorphism of a closed manifold with one-dimensional center bundle. Then the following statements are equivalent:
	\begin{enumerate}
		\item $f$ has quasi-isometric center bundle;
		\item there are constants $C\geq 1$ and $D>0$ such that for any $C^1$ curve $c$ tangent to $E^c$, we have that $C^{-1}\cdot L(c)-D\leq L(f^n(c))\leq C\cdot L(c)+D$ for any $n\in\Z$;
		\item there exists constants $r, R>0$ such that for any $x\in M$ and any $C^1$ curve $c(x)$ through $x$ tangent to $E^c$ of length $L(c)\leq r$, $f^n(c)$ belongs to a center curve $c'(f^n(x))$ through $f^n(x)$ of length $L(c'(f^n(x)))\leq R$ for any $n\in Z$.
	\end{enumerate}
\end{lem}
\begin{proof}
	It is obvious that (2) implies (1) by choosing $A=C$ and $B=D$. The implication from (1) to (2) follows by the choice of constants $C=A$ and $D=AB$. Item (1) can be deduced from (3) by taking $A=R/r$ and $B=R$. Notice that for any $C^1$ curve $c$ everywhere tangent to $E^c$, the curves $f(c)$ and $f^{-1}(c)$ are also everywhere tangent to $E^c$ as $E^c$ is an invariant bundle. By taking $r=1$ and $R=A(1+B)$, one can see that (1) implies (3). Therefore, three items are all equivalent.
\end{proof}

\begin{lem}\label{QIleaf=}
	Let $f: M\rightarrow M$ be a partially hyperbolic diffeomorphism of a closed manifold with one-dimensional center bundle. Assume $\mF^c$ is an $f$-invariant center foliation. Then the following statements are equivalent:
	\begin{enumerate}
		\item $f$ has quasi-isometric center foliation $\mF^c$;
		\item there are constants $C\geq 1$ and $D>0$ such that for any curve $c$ inside a leaf of $\mF^c$, we have that $C^{-1}\cdot L(c)-D\leq L(f^n(c))\leq C\cdot L(c)+D$ for any $n\in\Z$;
		\item there exist constants $r, R>0$ such that for any $x\in M$, we have that $f^n(W^c_r(x))\subset W^c_R(f^n(x))$ for any $n\in Z$, where $W^c_r(x)$ denotes the segment of center leaf $\mF^c(x)$ centered at $x$ of radius $r$.
	\end{enumerate}
\end{lem}
\begin{proof}
	The same argument in the proof of \lemref{QIcurve=} can be applied here by the fact that $\mF^c$ is an $f$-invariant foliation.
\end{proof}

Recall in \defref{QIdef-curve} that we do not require the existence of an invariant center foliation. However, the following result shows that the partially hyperbolic diffeomorphisms with quasi-isometric center bundle admit uniquel center foliations.

\begin{prop}\label{integrable}
	If $f: M\rightarrow M$ is a partially hyperbolic diffeomorphism with one-dimensional quasi-isometric center bundle, then $E^c$ is uniquely integrable. In particular, $f$ is dynamically coherent.
\end{prop}
\begin{proof}
	Suppose to the contrary that $E^c$ is not uniquely integrable. Then, there exist at least two curves tangent to $E^c$, denoted by $c_1$ and $c_2$, through a common point $x\in M$. Without loss of generality, we can assume that $c_1$ and $c_2$ are contained in the same $cu$-plaque through $x$. Indeed, if they are contained in the same $cs$-plaque, then we can apply our argument by considering $f^{-1}$. Otherwise, the $cu$-plaque through $c_1$ intersects the $cs$-plaque through $c_2$ in a curve tangent to $E^c$ that is distinct from $c_1$, and then we replace $c_2$ by this center curve. 
	
	Up to taking sub-curves, we assume that $c_1$ and $c_2$ are two short center curves with the same endpoint $x$ whose other endpoints $y_1\in c_1$ and $y_2\in c_2$ satisfy that $y_2\in W^u_{\delta}(y_1)$ for a small $\delta>0$. By the continuity of $E^c$, as $\delta$ can be chosen small enough, the intersection of center curves through $x$ with $W^u_{\delta}(y_1)$ is connected, denoted by $D\subset W^u_{\delta}(y_1)$. Denote by $c_z$ the center curve between $x$ and any point $z\in D$. Without loss of generality, we can assume that for any $z\in D$, the interior of $c_z$ is disjoint with $W^u_{\delta}(y_1)$. By the uniform expansion along unstable manifolds, for any $K>0$, there is an integer $N\in\N$ such that for any $n\geq N$, $f^n(y_2)$ is contained in $f^n(W^u_{\delta}(y_1))\backslash W^u_K(f^n(y_1))$. In particular, the center curve $f^n(c_2)$ is disjoint with $W^u_K(f^n(y_1))$ and intersects $f^n(W^u_{\delta}(y_1))$ in the unique point $f^n(y_2)$. As $f$ has quasi-isometric center bundle, there exists a uniform constant $R>0$ such that the length of $f^n(c_z)$ is bounded by $R$ for any $n\in \N$ and any $z\in D$. It turns out that the center curve $f^n(c_z)$ is entirely contained in the $cu$-ball $B^{cu}_R(f^n(x))$ centered at $f^n(x)$ of radius $R$ with respect to the induced metric for any $n\in \N$. Then, the curve $f^n(D)$ and, in particular, $f^n(c_2)$ are entirely contained in $B^{cu}_R(f^n(x))$ for any $n\in\N$. Taking $K$ going to infinity, by the transversality of $E^c$ and $E^u$, we can find a large enough $n\in \N$ such that $f^n(W^u_{\delta}(y_1))$ intersects the interior of $f^n(c_z)$ for some $z\in D$. This is a contradiction.
	
\end{proof}

\begin{rmk}
	We can actually deduce from \propref{integrable} that the quasi-isometric center bundle given by \defref{QIdef-curve} implies the existence of an invariant center foliation, which is actually unique. 
\end{rmk}

Let us revisit the two models of partially hyperbolic diffeomorphisms given in Section \ref{section-pre}. Notice that any partially hyperbolic diffeomorphism with a uniformly compact center foliation admits quasi-isometric center foliation. By this fact, as a consequence of \cite{DeMartinchich20}, we can obtain the quasi-isometric property for skew products.
\begin{prop}\label{QI-skewproduct}
	Let $f: M\rightarrow M$ be a skew product partially hyperbolic diffeomorphism of a closed manifold with $dimE^c=1$ and $dimE^u=1$. If $f$ is dynamically coherent, then it has quasi-isometric center foliation. If the center bundle $E^c$ is uniquely integrable, then $f$ has quasi-isometric center bundle.
\end{prop}

The result below follows directly from \cite[Proposition G.2]{BFFP1}, where it shows that every discretized Anosov flow is dynamically coherent.
\begin{prop}\label{QI-DAF}
	Let $f: M\rightarrow M$ be a discretized Anosov flow of a closed manifold. Then, it has a unique quasi-isometric center foliation. If $E^c$ is uniquely integrable, then $f$ has quasi-isometric center bundle.
\end{prop}

We should point out that, following directly from the definitions and \propref{integrable}, any partially hyperbolic diffeomorphism with one-dimensional quasi-isometric center bundle satisfies quasi-isometric center foliation which is actually unique. In other words, the central quasi-isometry given in \defref{QIdef-curve} is more restrictive than the one given in \defref{QIdef-leaf}. However, the inclusion in the reverse direction is false. Indeed, one can deduce from \propref{integrable} that any partially hyperbolic diffeomorphism with quasi-isometric center foliation whose center bundle $E^c$ is not uniquely integrable cannot admit quasi-isometric center bundle. Such diffeomorphisms actually exist even in dimension three, see \cite{2016example} for examples on $\T^3$ with uniformly compact center foliation and \cite{Martinchich23} for examples of discretized Anosov flows.

As stated above, we present the relation between two definitions of central quasi-isometries given by \defref{QIdef-leaf} and \defref{QIdef-curve} as below.
\begin{prop}\label{2QI}
	The set of partially hyperbolic diffeomorphisms with one-dimensional quasi-isometric center bundles belongs properly to the set of those with quasi-isometric center foliations.
\end{prop}

\subsection{Complete invariant foliations}
Denote by $\mF^c$ an invariant center foliation such that the partially hyperbolic diffeomorphism acts quasi-isometrically at $\mF^c$ as stated in \defref{QIdef-leaf}, which is the unique invariant foliation when the diffeomorphism has quasi-isometric center bundle. There are invariant foliations tangent to $E^s\oplus E^c$ and $E^c\oplus E^u$, denoted by $\mF^{cs}$ and $\mF^{cu}$, respectively, called the center-stable and center-unstable foliations. Denoted by $\mF^{cs}(x)$ (resp. $\mF^{cu}(x)$) the center-stable (resp. center-unstable) leaf through $x\in M$. We will denote by $W^{\sigma}(x)$ the corresponding leaf through the point $x\in M$ and $W^{\sigma}(U)\colon=\bigcup_{y\in U}W^{\sigma}(y)$ for any subset $U\subset M$ and $\sigma=s, c, u$.

\begin{defn}
	We say that the foliation $\mF^{cs}$ is \emph{complete} if for any point $x\in M$, we have that $\mF^{cs}(x)=W^s(W^c(x))$. 
\end{defn}

A priori, the definition above only tells us that each center stable leaf is an $s$-saturated set of any center leaf. In fact, the completeness of $\mF^{cs}$ implies that each center stable leaf is bifoliated by center and stable leaves.

\begin{lem}\cite[Lemma 3.3]{FU1}
	The foliation $\mF^{cs}$ is complete if and only if $\mF^{cs}(x)=W^s(W^c(x))=W^c(W^s(x))$ for any point $x\in M$.
\end{lem}

The lemma below follows directly from the continuity of center bundle $E^c$ and center foliation.
\begin{lem}\label{continuity}
	Let $\mF^c$ be a center foliation tangent to $E^c$. For any $K>0$ and any $\epsilon>0$, there is $\delta=\delta(K, \epsilon)>0$ such that for any center curve $\gamma$ in a leaf of $\mF^c$ of length less than $K$ and any point $z\in W^s_{\delta}(\gamma)$, the connected center curve $ c(z)\subset W^s_{\epsilon}(\gamma)\cap \mF^c(z)$ through point $z$ has non-empty intersection with the local stable manifold $W^s_{\epsilon}(w)$ for any $w\in \gamma$. The analogous statement is true if we replace stable manifolds with unstable ones.
\end{lem}

The following result states that the two types of central quasi-isometry given in \defref{QIdef-leaf} and \defref{QIdef-curve} have the same foliated information in each leaf of center stable and center unstable foliations as long as the invariant center foliation with quasi-isometric property is provided. It does not matter in following proposition whether the invariant center bundle $E^c$ is uniquely integrable or not.
\begin{prop}\label{complete}
	If $f: M\rightarrow M$ is a partially hyperbolic diffeomorphism with one-dimensional quasi-isometric center bundle or quasi-isometric center foliation, then the invariant foliations $\mF^{cs}$ and $\mF^{cu}$ are complete. Moreover, the lifted foliations $\wt{\mF}^{cs}$ and $\wt{\mF}^{cu}$ are also complete.
\end{prop}
\begin{proof}
	We will only prove the completeness of $\mF^{cs}$ as the completeness of $\mF^{cu}$ can be shown in a similar way. Suppose by contradiction that there is a point $x\in M$ such that $W^s(W^c(x))$ is a strict subset of $\mF^{cs}(x)$.  As shown in \cite{BW05}, we are able to find a point $y\in \mF^{cs}(x)\backslash W^s(W^c(x))$ and a center curve $\gamma:[0, 1]\rightarrow M$ such that $\gamma(1)=y$ and $\gamma(t)\in W^s(W^c(x))$ for any $t\in \left[0,1\right)$. It follows that the stable leaf $W^s(y)$ is disjoint with $W^c(x)$. By the assumption that $f$ has quasi-isometric center bundle, there are constants $A\geq 1$ and $B>0$ such that the length $L(f^n(\gamma))$ is uniformly bounded by the constant $A\cdot L(\gamma)+B$ for any $n\in \N$, where $L(\gamma)$ is the length of $\gamma$. Take any small $\epsilon>0$. By \lemref{continuity}, there exists $\delta>0$ only depending on the choice of $\epsilon$ and the given number $A\cdot L(\gamma)+B$ such that if $W^s_{\delta}(f^n(\gamma))$ intersects the center leaf $W^c(f^n(x))$ for some integer $n\in\N$, then $W^s_{\epsilon}(f^n(y))$ intersects $W^c(f^n(x))$ as well. Note that the point $\gamma(0)$ is contained in the $s$-saturated set of $W^c(x)$. Then, there exists $N\in \N$ so that for any $n\geq N$, the point $f^n(\gamma(0))$ is contained in $W^s_{\delta}(W^c(x))$. Therefore, we conclude that the local stable manifold $W^s_{\epsilon}(f^n(y))$ has non-empty intersection with the leaf $W^c(f^n(x))$ for any $n\geq N$. However, the stable leaf $W^s(f^n(y))$ is disjoint with $W^c(f^n(x))$ for any $n\in \Z$ since $W^s(y)$ is disjoint with $W^c(x)$. Thus, we arrive at a contradiction and then complete the proof.
	
	The completeness of lifted foliations $\wt{\mF}^{cs}$ and $\wt{\mF}^{cu}$ follows by the same argument.
\end{proof}

We have the following consequence using the analogous argument in \cite{BW05, Zhang21}.

\begin{cor}\label{cylinder/plane}
	Let $f: M\rightarrow M$ be a partially hyperbolic diffeomorphism with one-dimensional quasi-isometric center bundle or quasi-isometric center foliation. Then, the leaves of $\mF^{cs}$ and $\mF^{cu}$ have abelian fundamental groups. In particular, if the manifold $M$ is 3-dimensional, then up to a double cover, each leaf of $\mF^{cs}$ and $\mF^{cu}$ is either a cylinder or a plane. A leaf is a cylinder if it contains a compact center leaf.
\end{cor}
\begin{proof}
	By \propref{integrable} or \defref{QIdef-leaf}, the partially hyperbolic diffeomorphism $f$ is dynamically coherent. There are $f$-invariant foliations $\mF^{cs}$ and $\mF^{cu}$ tangent to $E^s\oplus E^c$ and $E^c\oplus E^u$, respectively, both of which are complete by \propref{complete}. It implies that, by partial hyperbolicity, there is no compact stable leaf. Then, for each center stable leaf, the fundamental group acts freely on the leaf space of stable foliation, which is homeomorphic to $\R$ as the center dimension is one. Thus, each leaf has abelian fundamental group \cite{Holder1901}.
	
	In the three-dimensional case, up to a finite lift, we can assume that each leaf of $\mF^{cs}$ is orientable. As $\mF^{cs}$ is foliated by surfaces, each leaf is either a cylinder, a plane, or a torus. However, the dynamical coherence implies the absence of compact center stable leaves \cite{2015Center}, and then excludes the presence of torus. 
	
	Let $\gamma\in \mF^c$ be a compact center leaf and $\mF^{cs}(\gamma)\in \mF^{cs}$ be the center stable leaf through $\gamma$. Suppose that $\mF^{cs}(\gamma)$ is a plane. If there is a stable leaf in $\mF^{cs}(\gamma)$ intersecting $\gamma$ in at least two points, then there is a loop in $\mF^{cs}(\gamma)$ by connecting a stable segment and a segment of $\gamma$. This loop encloses a disk since $\mF^{cs}(\gamma)$ is a plane, which gives a contradiction by Poincar\'{e}-Bendixon Theorem. If every stable leaf in $\mF^{cs}(\gamma)$ intersects $\gamma$ in a unique point, then, again by Poincar\'{e}-Bendixon Theorem, we also arrive to a contradiction in the disk bounded by $\gamma$. Therefore, the leaf $\mF^{cs}(\gamma)$ can only be a cylinder.
	
	Thus, we conclude the result for $\mF^{cs}$. The statement for $\mF^{cu}$ can be obtained in the analogous way.
\end{proof}

%% file: homoclinic.tex
% !Mode:: "TeX:UTF-8"

\section{Periodic center leaves}\label{section-periodic}
In the subsequent sections, we will primarily refer to a partially hyperbolic diffeomorphism with quasi-isometric center as defined in \defref{QIdef-leaf}. However, the discussion in this section applies equally to both notions of central quasi-isometries given in \defref{QIdef-leaf} and \defref{QIdef-curve}. For the sake of simplicity, we will state the results explicitely only for quasi-isometric center (foliation).

Let $M$ be a closed 3-dimensional manifold and consider a partially hyperbolic diffeomorphism $f: M\rightarrow M$ with quasi-isometric center. As established in \propref{integrable}, there exist invariant foliations $\mF^{cs}$, $\mF^{cu}$ and $\mF^c$, tangent to $E^s\oplus E^c$, $E^c\oplus E^u$ and $E^c$, respectively. Under the existence of these foliations, we are going to examine the presence of periodic compact center leaves and their interactions with nearby center leaves within the corresponding center-stable and center-unstable leaves.

\subsection{Existence of periodic compact center leaves}

In this subsection, we are going to explore the existence of a periodic compact center leaf. Recall in \corref{cylinder/plane} that compact center leaves can only exist in cylinders. So, our strategy is to study the center leaves in center stable and center unstable cylinder leaves.

An extreme situation is that the center stable foliation $\mF^{cs}$ has no cylinder leaf. However, the following result explicitly reveals that, in this situation, the ambient 3-manifold can be determined. 
\begin{thm}\cite{Rosenberg,Gabai90}\label{rosenberg}
	If a closed 3-manifold admits a foliation by planes, then it is the 3-torus.
\end{thm}

We first note that any non-compact center leaf in a center stable cylinder must have infinitely many distinct intersection points with every stable leaf, see the lemma below. As we will see in \lemref{centersaturated}, the compact center leaves have a totally different nature. 

The following lemma is indicated by an argument appeared in \cite[Lemma 3.6]{Zhang21} concerned with the topologically neutral center setting. We present the proof here in the context of quasi-isometric center.

\begin{lem}\label{twice}
	Let $F\in \mF^{cs}$ be a cylinder leaf. If $l\subset F$ is a non-compact center leaf, then for any point $y\in F$, the stable leaf $W^s(y)$ intersects $l$ in infinitely many points.
\end{lem}
\begin{proof}
	Let $\tilde{F}\in\wt{\mF}^{cs}$ be a lifted leaf in the universal cover $\widetilde{M}$. Denote by $\tilde{l}\subset\tilde{F}$ a lift of $l$ and $\tilde{y}\in \tilde{F}$ a lift of $y$ . By the completeness of $\wt{\mF}^{cs}$ given in \propref{complete}, the leaf $\tilde{l}$ intersects the stable leaf $\wt{W}^s(\tilde{y})$ in a point, say $z\in \tilde{F}$. As $F$ is a cylinder leaf, the group of deck transformations fixing $\tilde{F}$ has a non-trivial generator $\tau\in\pi_1(F)$. For each $i\in\Z$, $\tau^i(\wt{W}^s(\tilde{y}))$ is a stable leaf in $\tilde{F}$ distinct from $\wt{W}^s(\tilde{y})$, since there is no closed stable leaf in $M$. By the completeness again, the leaves $\tau^i(\wt{W}^s(\tilde{y}))$ and $\tilde{l}$ have non-empty intersection for each $i\in\Z$. Denote by $z_i=\tau^i(\wt{W}^s(\tilde{y}))\cap\tilde{l}$, which is the unique point by Poincar\'{e}-Bendixon Theorem. By the fact that $l$ is non-compact, the points $z_i\in\tilde{l}$ are pairwise distinct and also distinct from $z$. Thus, by projecting $z_i$ under the universal covering map, the leaf $W^s(y)$ has infinitely many intersection points with $l$.
\end{proof}

The following proposition is a variation of Bowen's shadowing lemma for Axiom A diffeomorphisms. While the classical shadowing lemma implies the existence of periodic points, the result below produces the existence of periodic compact center leaves. We refer the reader to \cite{Bowen_book, Carrasco15compact} and \cite[Appendix A]{BZ20} for the details.
\begin{prop}\cite[Proposition A.1]{BZ20}\label{Bowen}
	Let $f: M\rightarrow M$ be a dynamically coherent partially hyperbolic diffeomorphism of a closed 3-manifold and $\Lambda\subset M$ be a compact center saturated set. If the center leaves in $\Lambda$ are uniformly compact, then for any $\epsilon>0$, there exists a periodic compact center leaf in the $\epsilon$-neighborhood of $\Lambda$.
\end{prop}

Now, we are ready to show our main result of this subsection.
\begin{prop}\label{periodiccompact}
	Let $f: M\rightarrow M$ be a partially hyperbolic diffeomorphism with quasi-isometric center in a closed 3-manifold. Then, there is at least one $f$-periodic compact center leaf.
\end{prop}

\begin{proof}
	We first suppose that all leaves of $\mF^{cs}$ are planes. As shown in \thmref{rosenberg}, the manifold $M$ is the 3-torus. By the quasi-isometric property in the center, it is shown in \cite{Hammerlindl13} that the diffeomorphism $f$ is (leaf) conjugate to a skew product over an Anosov diffeomorphism on the 2-torus. In particular, it has non-planar center stable leaves, which is absurd. Then, there exists at least one leaf of $\mF^{cs}$ that is not a plane.
	
	It follows from \corref{cylinder/plane} that there is at least one cylinder leaf, denoted by $F\in \mF^{cs}$. Suppose that $F$ contains a compact center leaf, denoted by $\gamma$. Let $\Lambda\subset M$ be the closure of the orbit $\{f^n(\gamma)\}_{n\in\Z}$, which is an $f$-invariant center saturated set. By the definition of quasi-isometric center, see \defref{QIdef-leaf} and \lemref{QIleaf=}, the center leaves in $\Lambda$ are uniformly compact of lengths bounded by $A\cdot L(\gamma)+B$ for some constants $A\geq 1$ and $B>0$.
	By \propref{Bowen}, there exists a periodic compact center leaf in an arbitrarily small neighborhood of the compact set $\Lambda$. 
	
	If the center leaves in $F$ are all non-compact, then we can construct a loop, denoted by $\alpha\subset F$, by connecting a center curve with a stable curve. Such a loop always exists due to \lemref{twice}. As the loop $\alpha$ is compact and each stable curve is uniformly contracting, we deduce that the length $L(f^n(\alpha))$ is uniformly bounded by $A\cdot L(\alpha)+B$ for all $n\in \N$, where $L(\alpha)$ denotes the length of the curve $\alpha$. It turns out by the compactness of $M$ that the sequence of curves $f^n(\alpha)$ accumulates in a compact curve $\beta$ of length also bounded by $A\cdot L(\alpha)+B$. Moreover, the curve $\beta$ is a compact center leaf by the construction of $\alpha$. By the quasi-isometric center property, the center leaves in the orbit of $\beta$ are uniformly compact. Again using \propref{Bowen}, there exists a periodic compact center in a small neighborhood of the closure of orbit of $\beta$. Thus, we finish the proof.
\end{proof}

\subsection{Homoclinic intersection of periodic center leaves}

We are going to recall a notion, called $N$ homoclinic point, for surface diffeomorphisms with respect to a hyperbolic saddle given in \cite{BonattiLangevin98} and introduce its generalization for partially hyperbolic diffeomorphisms with quasi-isometric center with respect to a periodic compact center leaf. The techniques in this subsection were essentially developed in \cite{BZ20} in which the authors are concerned with diffeomorphisms admitting topologically neutral center.

\begin{defn}
	Let $f: S\rightarrow S$ be a $C^1$ diffeomorphism of a surface without homoclinic tangencies. Given a hyperbolic saddle point $p\in S$, we say that a point $x\in W^s(p)\cap W^u(p)\backslash\{p\}$ is an \emph{$N$ homoclinic point} of $p$ if the stable segment $s_p(x)$ between $p$ and $x$ intersects the unstable segment $u_p(x)$ between $p$ and $x$ in $N$ points. In particular, we can define a function $n: W^s(p)\cap W^u(p)\backslash\{p\}\rightarrow \N$ such that $x$ is an $n(x)$ homoclinic point of $p$. We call the integer $n(x)$ the \emph{homoclinic intersection number} of $x$ with respect to $p$.
\end{defn}

\begin{prop}\cite[Proposition 6.3]{BZ20}
	Let $f: S\rightarrow S$ be a $C^1$ diffeomorphism of a surface without homoclinic tangencies and $p$ be a hyperbolic saddle fixed point. The homoclinic intersection number is $f$-invariant, i.e., $n(x)=n(f^n(x))$ for any $x\in W^s(p)\cap W^u(p)\backslash\{p\}$ and any $n\in \Z$. Moreover, for each $N\in\N$, there are finitely many orbits $O(x)\colon=\{f^n(x)\}_{n\in \Z}$ with homoclinic intersection number $n(f^n(x))=n(x)\leq N$.
\end{prop}

We are going to generalize this notion to the context of partially hyperbolic diffeomorphisms with quasi-isometric center. Before doing that, let us provide the following lemma to ensure that the generalization of homoclinic intersection points to center leaves is well-defined.

\begin{lem}\label{centersaturated}
	Let $f:M\rightarrow M$ be a partially hyperbolic diffeomorphism with two complete invariant foliations $\mF^{cs}$ and $\mF^{cu}$ and $\gamma$ be an $f$-periodic compact center leaf. We have the following:
	\begin{itemize}
		\item The center leaf $\gamma$ intersects any stable leaf in $\mF^{cs}(\gamma)$ in a unique point and intersects any unstable leaf in $\mF^{cu}(\gamma)$ in a unique point.
		\item For any point $x\in \mF^{cs}(\gamma)\cap \mF^{cu}(\gamma)$, we have that the center leaf $\mF^c(x)$ through $x$ is entirely contained in $W^s(\gamma)\cap W^u(\gamma)$.
	\end{itemize}
\end{lem}
\begin{proof}
	By the completeness of $\mF^{cs}$, the leaf $\mF^{cs}(\gamma)$ through $\gamma$ is exactly the same set as $W^s(\gamma)$. It implies that $\gamma$ has non-empty intersection with any stable leaf in $\mF^{cs}(\gamma)$. Suppose that there are two points $p\neq q\in \gamma$ such that $p\in W^s(q)$. Then, for any small $\epsilon>0$, there exists $N\in \N$ so that for any $n\geq N$, $f^n(p)$ is contained in the local stable manifold $W^s_{\epsilon}(f^n(q))$. Since $f^n(p)$ and $f^n(q)$ are contained in the same compact center leaf $\gamma$, we may assume up to a subsequence that $f^n(p)$ and $f^n(q)$ converge to a common point in $\gamma$. Then we lose the transversality of $E^s$ and $E^c$ at this limit point, which is a contradiction. The same holds for $\mF^{cu}(\gamma)$.
	
	By the completeness, we have that $\mF^{cs}(\gamma)\cap \mF^{cu}(\gamma)$ coincides with $W^s(\gamma)\cap W^u(\gamma)$. As each leaf of $\mF^{cs}$ and $\mF^{cu}$ is $\mF^c$-saturated, the leaf $\mF^c(x)$ is a subset of $W^s(\gamma)\cap W^u(\gamma)$ for any point $x\in \mF^{cs}(\gamma)\cap \mF^{cu}(\gamma)$. 
\end{proof}

For any point $x\in \mF^{cs}(\gamma)\cap \mF^{cu}(\gamma)\backslash\gamma$, as shown in \lemref{centersaturated}, the leaf $W^s(x)$ intersects $\gamma$ in a unique point, denoted by $x_s=\gamma\cap W^s(x)$. If the center leaf $\mF^c(x)$ is not compact, then it intersects the stable leaf $W^s(x)$ in infinitely many points by \lemref{twice}. To give an order on these intersection points, we denote by $\{x_s^i\}_{i\in \Z}$ all points in $\mF^c(x)\cap W^s(x)$ satisfying that 
\begin{itemize}
	\item $x_s^0=x$;
	\item the center segment $[x_s^i, x_s^{i+1}]_c$ between $x_s^i$ and $x_s^{i+1}$ does not contain any other points in $\mF^c(x)\cap W^s(x)$;
	\item the stable segment $s_{x_s}(x_s^i)$ between $x_s$ and $x_s^i$ belongs to the stable segment $s_{x_s}(x_s^j)$ between $x_s$ and $x_s^j$ for any $j>i$.
\end{itemize}
For each $i\in \Z$, the stable segment $s_{x_s}(x_s^{i+1})$, center segment $[x_s^i, x_s^{i+1}]_c$ and $\gamma$ enclose a compact connected region in $\mF^{cs}(x)$, denoted by $s_{\gamma}(x_s^i)$. In particular, we denote by $s_{\gamma}(x)\colon=s_{\gamma}(x_s^0)$. Analogously, we define a sequence of nested compact connected regions $u_{\gamma}(x_u^i)$ in the leaf $\mF^{cu}(x)$ for points $\{x_u^i\}\subset \mF^c(x)\cap W^u(x)$ and $u_{\gamma}(x)\colon=u_{\gamma}(x_u^0)$. 

When the center leaf $\mF^c(x)$ is compact, the intersection points $\{x_s^i\}$ are actually the exact single point $x$ by the same proof of \lemref{centersaturated}. In this case, one can still discuss about the homoclinic intersection center and the intersection number.

\begin{defn}(Homoclinic intersection center)
	Let $f:M\rightarrow M$ be a partially hyperbolic diffeomorphism with two complete invariant foliations $\mF^{cs}$ and $\mF^{cu}$. For an $f$-periodic compact center leaf $\gamma$, we say that a center leaf $c\subset W^s(\gamma)\cap W^u(\gamma)\backslash\{\gamma\}$ is an \emph{$N$ homoclinic center} of $\gamma$ if there is a point $x\in c$ such that the compact regions $s_{\gamma}(x)$ and $u_{\gamma}(x)$ defined above intersect in $N$ connected center curves. 
	We define a function $n: W^s(\gamma)\cap W^u(\gamma)\backslash\{\gamma\}\rightarrow \N$ such that $c$ is an $n(c)$ homoclinic center of $\gamma$. We call the integer $n(c)$ the \emph{homoclinic intersection number} of $c$ with respect to $\gamma$.
\end{defn}

As shown in \lemref{centersaturated}, the intersection $W^s(\gamma)\cap W^u(\gamma)$ is center saturated. One can easily see that the definition above is independent of the choice of $x$ in the center leaf $c$. In other words, the homoclinic intersection number $n(c)$ will not change if we replace the point $x\in c$ by any other point in $c$. So, we can define this number in the set of center leaves.

Besides, one can see in \lemref{centersaturated} that the notion above is actually well-defined for any center leaf in $\mF^{cs}(\gamma)\cap \mF^{cu}(\gamma)$ which is, a priori, a larger set containing $W^s(\gamma)\cap W^u(\gamma)$ if we do not assume the completeness of two invariant foliations.

\begin{lem}\label{f-inv}
	Let $\gamma$ be an $f$-periodic compact center leaf. The homoclinic intersection number is a constant in the orbit of any leaf $c\in \mF^{cs}(\gamma)\cap \mF^{cu}(\gamma)$, that is, $n(c)=n(f^n(c))$ for any $n\in\Z$.
\end{lem}
\begin{proof}
	It directly follows from the $f$-invariance of center and stable foliations. More precisely, for any center leaf $c\in\mF^{cs}(\gamma)\cap \mF^{cu}(\gamma)$ and any point $x\in c$, we have $s_{f^n(\gamma)}(f^n(x))=f^n(s_{\gamma}(x))$ and $u_{f^n(\gamma)}(f^n(x))=f^n(u_{\gamma}(x))$ for any $n\in \Z$. Thus, we have $n(c)=n(f^n(c))$ for any $n\in\Z$.
\end{proof}

Note that in the definition and lemma above, we allow the situation that there are compact center leaves in $\mF^{cs}(\gamma)$ or $\cap\mF^{cu}(\gamma)$. However, in the sequal discussion, we are going to ignore the existence of compact leaves in $\mF^{cs}(\gamma)\cup\mF^{cu}(\gamma)$. The following result can be deduced by adapting the same argument as in \cite[Proposition 6.8]{BZ20}. We provide a sketch here for the sake of completeness.

\begin{prop}\label{finitecenter}
	Let $f:M\rightarrow M$ be a partially hyperbolic diffeomorphism with two complete invariant foliations $\mF^{cs}$ and $\mF^{cu}$ and $\gamma$ be an $f$-periodic compact center leaf. Assume that there is no other compact center leaves in $\mF^{cs}(\gamma)\cup\mF^{cu}(\gamma)$. Then, for each $N\in \N$, there are finitely many center leaves whose homoclinic intersection numbers are bounded by $N$.
\end{prop}
\begin{proof}
	For any $N\in \N$, we let $c\in \mF^{cs}(\gamma)\cap\mF^{cu}(\gamma)\backslash\{\gamma\}$ be a center leaf so that $n(c)$ is the largest homoclinic intersection number bounded by $N$. Let $c'\in \mF^{cs}(\gamma)\cap\mF^{cu}(\gamma)\backslash\{\gamma\}$ be any center leaf of $n(c')\leq N$ distinct from $c$. As there is no compact leaves in $\mF^{cs}(\gamma)\cup\mF^{cu}(\gamma)$ unless $\gamma$, by the Poincar\'{e}-Bendixon Theorem, both $c$ and $c'$ accumulate in the center leaf $\gamma$. Take a point $x\in c$ and a point $y\in c'$. Since the homoclinic intersection number is independent of the choice of points, we can assume that $y$ is contained in the stable segment $s_{x_s}(x_s^1)\backslash s_{x_s}(x)$ between $x=x_s^0$ and $x_s^1$, which implies that $s_{\gamma}(x)\subset s_{\gamma}(y)$.
	
	We claim that the point $y$ must be contained in $u_{\gamma}(x_u^1)$. Otherwise, if $y$ is outside of $u_{\gamma}(x_u^1)$, then by definition we have that $u_{\gamma}(x)\subset u_{\gamma}(y)$. It implies that $s_{\gamma}(y)\cap u_{\gamma}(y)$ has more connected center curves than $s_{\gamma}(x)\cap u_{\gamma}(x)$ since $c\neq c'$.In other words, we have $n(c')\geq n(c)+1$, which contradicts to our choice of $c$. 
	
	Thus, the point $y$ is contained in the intersection $s_{\gamma}(x)\cap u_{\gamma}(x_u^1)$. As both $s_{\gamma}(x)$ and $u_{\gamma}(x_u^1)$ are compact, by the transversality of $\mF^{cs}$ and $\mF^{cu}$, there are finitely many connected center curves. We finish the proof.
\end{proof}

%% file: classify_transitive.tex
% !Mode:: "TeX:UTF-8"

\section{Classification under transitivity}\label{section-classify-transitive}

In this section, we are going to present a classification of 3-dimensional transitive partially hyperbolic diffeomorphisms with quasi-isometric center, addressing the original Pujals conjecture. 

The following theorem is the main result of this section:
\begin{thm}\label{classify_transitive}
	Let $f: M^3\rightarrow M^3$ be a $C^1$ transitive partially hyperbolic diffeomorphism of a closed 3-manifold with quasi-isometric center. Then, either 
	\begin{itemize}
		\item up to a finite lift and iterate, $f$ is conjugate to a skew product over an Anosov automorphism on $\mathbb{T}^2$; or
		\item $f$ admits an iterate that is a discretized Anosov flow.
	\end{itemize}
\end{thm}

\begin{rmk}
	As mentioned before, the classification result in \thmref{classify_transitive} is valid for both central quasi-isometries given by \defref{QIdef-leaf} and \defref{QIdef-curve}. Here, for the sake of simplicity, we will only present the proof of \thmref{classify_transitive} for quasi-isometric center given by \defref{QIdef-curve}, and the other one follows by exactly the same argument.
\end{rmk}

Our strategy for proving the theorem above is to utilize the following result in the context of a partially hyperbolic diffeomorphism with quasi-isometric center.

\begin{thm}\cite{BW05}\label{BW}
	Let $f: M\rightarrow M$ be a $C^1$ dynamically coherent partially hyperbolic diffeomorphism of a closed 3-manifold. 
	\begin{itemize}
		\item Assume that $f$ is transitive. If there exists a compact periodic center leaf $\gamma\in\mF^c$ and a positive number $\delta>0$ such that $W^s_{\delta}(\gamma)\cap W^u_{\delta}(\gamma)\backslash\{\gamma\}$ contains a compact center leaf, then up to a finite lift and iterate, $f$ is conjugate to a skew product over a linear Anosov of the torus.
		\item If there exists a compact periodic center leaf $\gamma\in \mF^c$ such that every center leaf in $W^s_{loc}(\gamma)$ is $f$-periodic, then all center leaves are $f$-periodic of a common period and the center foliation supports a continuous flow conjugate to a transitive expansive flow.
	\end{itemize}
\end{thm}

We would like to establish some technical assumptions on the existence of particular compact center leaves to classify partially hyperbolic diffeomorphisms with quasi-isometric center.
\begin{prop}\label{skewproduct}
	Let $f: M\rightarrow M$ be a transitive partially hyperbolic diffeomorphism with quasi-isometric center of a closed 3-manifold. Assume that $\gamma$ is an $f$-periodic compact center leaf. If there is a compact center leaf in $\mF^{cs}(\gamma)\backslash\{\gamma\}$ or $\mF^{cu}(\gamma)\backslash\{\gamma\}$, then up to a finite lift and iterate, $f$ is conjugate to a skew product over a linear Anosov automorphism on the torus.
\end{prop}
\begin{proof}
	We only need to show it in the case that there is another compact center leaf in the center stable leaf through $\gamma$, while the center unstable case follows symmetrically. Assume that there exists a compact center leaf $c\in \mF^{cs}(\gamma)\backslash\{\gamma\}$. Without loss of generality, we assume that $\gamma=f(\gamma)$ for simplicity. By \corref{cylinder/plane}, the leaf $\mF^{cs}(\gamma)$ is a cylinder and the leaves $\gamma$ and $c$ bound a compact region in $\mF^{cs}(\gamma)$, denoted by $U$. By the Poincar\'{e}-Bendixon Theorem, every center leaf $\mF^c(x)\subset U$ either is compact or accumulates in two compact center leaves in $U$. We can assume that $\gamma$ and $c$ belong to the closure $\overline{\mF^c(x)}$ for each $x\in U$. By transitivity of $f$, there is a point $x\in U$ with dense backward orbit, that is, its limit set is the whole manifold $\alpha(x)=M$. Then, for any small $\epsilon>0$, there exists $N_{\epsilon}\in \N$ such that $f^{-N_{\epsilon}}(x)$ is contained in the $\epsilon$-neighborhood $U_{\epsilon}(\gamma)$ of $\gamma$. By transversality, for $\epsilon$ sufficiently small, we have that $\mF^{cs}_{loc}(f^{-N_{\epsilon}}(x))\cap \mF^{cu}_{loc}(\gamma)\neq \emptyset$. Pick a point $y\in \mF^{cs}_{loc}(f^{-N_{\epsilon}}(x))\cap \mF^{cu}_{loc}(\gamma)$, which is contained in $f^{-N_{\epsilon}}(U)$. It implies that the closure of center leaf $\overline{\mF^c(y)}$ is a subset of $U$. Moreover, by the Poincar\'{e}-Bendixon Theorem, there is a compact center leaf in $W^u(\gamma)\backslash\gamma$ accumulated by $\mF^c(y)$, which is actually the leaf $f^{-N_{\epsilon}}(c)$. It turns out that $f^{-N_{\epsilon}}(c)$ is a compact center leaf in the intersection of $f^{-N_{\epsilon}}(U)\subset \mF^{cs}(\gamma)$ and a compact center unstable region in $\mF^{cu}(\gamma)$. Applying \thmref{BW}, we conclude that, up to a finite lift and iterate, $f$ is conjugate to a skew product. Thus, we finish the proof.
\end{proof}

\begin{prop}\label{DAF}
	Let $f: M\rightarrow M$ be a transitive partially hyperbolic diffeomorphism with quasi-isometric center of a closed 3-manifold. Assume that $\gamma$ is an $f$-periodic compact center leaf. If there is no compact center leaf in $\mF^{cs}(\gamma)\backslash\{\gamma\}$ nor $\mF^{cu}(\gamma)\backslash\{\gamma\}$, then an iterate of $f$ is a discretized Anosov flow.
\end{prop}
\begin{proof}
	By \propref{integrable} and \propref{complete}, the diffeomorphism $f$ with quasi-isometric center admits two invariant foliations $\mF^{cs}$ and $\mF^{cu}$, both of which are complete. As shown in \propref{finitecenter}, for each integer $N\in\N$, there are finitely many center leaves in $\mF^{cs}(\gamma)\cap\mF^{cu}(\gamma)\backslash\{\gamma\}$ whose homoclinic intersection numbers with respect to $\gamma$ are bounded by $N$. By \lemref{f-inv}, for each center leaf in $\mF^{cs}(\gamma)\cap\mF^{cu}(\gamma)\backslash\{\gamma\}$, the homoclinic intersection number is preserved under iterations. It turns out that every leaf of homoclinic intersection number less than $N$ is $f$-periodic. As $N$ is arbitrarily chosen, we conclude that every leaf in $\mF^{cs}(\gamma)\cap\mF^{cu}(\gamma)$ is $f$-periodic.
	
	Since $f$ is transitive, the intersection $\mF^{cs}(\gamma)\cap\mF^{cu}(\gamma)$ is dense in $\mF^{cs}(\gamma)$. Note that the leaf $\mF^{cs}(\gamma)$ is a cylinder by \corref{cylinder/plane} which splits into two connected components by the center leaf $\gamma$. For each component, the center leaf space $\mathcal{L}^s$ is homeomorphic to a circle on which the diffeomorphism $f$ induces a homeomorphism $\hat{f}: \mathcal{L}^s\rightarrow \mathcal{L}^s$. By the fact that every leaf in $\mF^{cs}(\gamma)\cap\mF^{cu}(\gamma)$ is $f$-periodic, the induced map $\hat{f}$ has dense periodic points. It turns out that every point in $\mathcal{L}^s$ is $\hat{f}$-periodic, and thus every center leaf in $\mF^{cs}(\gamma)$ is $f$-periodic. Therefore, we finish the proof by applying \thmref{BW}.
\end{proof}

Now, let us provide the proof of \thmref{classify_transitive}.

\begin{proof}[Proof of \thmref{classify_transitive}]
	Consider any transitive partially hyperbolic diffeomorphism $f: M\rightarrow M$ with quasi-isometric center of a closed 3-dimensional manifold. By \propref{integrable}, there exists two invariant foliation $\mF^{cs}$ and $\mF^{cu}$ tangent to $E^s\oplus E^c$ and $E^c\oplus E^u$, respectively. There exists a periodic compact center leaf as shown in \propref{periodiccompact}, denoted by $\gamma$. If the center stable leaf $\mF^{cs}(\gamma)$ of $\gamma$ contains another compact center leaf, then up to a finite lift and iterate, $f$ is conjugate to a skew product by \propref{skewproduct}. The same holds if the center unstable leaf $\mF^{cu}(\gamma)$ contains another compact center leaf. Otherwise, neither $\mF^{cs}(\gamma)\backslash\{\gamma\}$ nor $\mF^{cu}(\gamma)\backslash\{\gamma\}$ contains a compact center leaf, which implies that some iterate of $f$ is a discretized Anosov flow by \propref{DAF}. Thus, we complete the proof.
\end{proof}

%% file: accessible.tex
% !Mode:: "TeX:UTF-8"

\section{Accessibility and ergodicity in dimension three}\label{section-accessible}
In this section, we will investigate the accessbility and ergodicity of 3-dimensional partially hyperbolic diffeomorphisms. This study is independent of the preceding sections and holds its own interest.

Let $f: M\rightarrow M$ be a $C^1$ partially hyperbolic diffeomorphism of a closed 3-manifold. It is well-known that the extreme strong bundles $E^s$ and $E^u$ in the partially hyperbolic splitting are both uniquely integrable \cite{BrinPesin74, PughShub72}. In particular, there exist unique invariant foliations $\mF^s$ and $\mF^u$ tangent to $E^s$ and $E^u$, respectively, assembled by $C^1$ injectively immersed submanifolds.

A set is \emph{s-saturated (resp. u-saturated)} if it is a union of stable (resp. unstable) leaves over all points in it. We say that a set is \emph{su-saturated} if it is both s- and u-saturated. For any point $x\in M$, the accessibility class $AC(x)$ of $x$ is defined as the minimal su-saturated set containing $x$. In other words, any pair of points lying in the same accessibility class can be joined by an path piecewise tangent to $E^s$ and $E^u$. We say that $f$ is \emph{accessible} if there is only one accessibility class which is equal to the whole manifold $M$. We will denote by $\Gamma(f)$ the set of non-open accessibility classes, which is a compact invariant set laminated by accessibility classes, see also \cite{08invent} for more discussion on this set.

We are going to show the following result in this section.
\begin{thm}\label{accessible-complete}
	Let $f: M^3\rightarrow M^3$ be a partially hyperbolic diffeomorphism of a closed 3-manifold with $NW(f)=M$. Assume that $f$ admits two complete invariant foliations $\mF^{cs}$ and $\mF^{cu}$ and $\pi_1(M)$ is not virtually solvable. Then $f$ is accessible.
\end{thm}

\subsection{Accessibility in a general case}
The following results provides a discription on accessibility classes for a non-wandering partially hyperbolic diffeomorphism.
\begin{thm}\cite{2008nil}\label{lamination}
	Let $f: M\rightarrow M$ be a partially hyperbolic diffeomorphism of a closed orientable 3-manifold such that $NW(f)=M$. Assume that three invariant bundles $E^s$, $E^c$, $E^u$ are all orientable and $f$ is not accessible. Then, exactly one of the following possibilities holds:
	\begin{itemize}
		\item there is an incompressible torus tangent to $E^s\oplus E^u$;
		\item there is an $f$-invariant lamination $\emptyset\neq \Gamma(f)\neq M$ tangent to $E^s\oplus E^u$ that trivially extends to a foliation without compact leaves;
		\item there is an invariant foliation $\Gamma(f)$ tangent to $E^s\oplus E^u$ without compact leaves.
	\end{itemize}
\end{thm}

Note that in the second situation above, the foliation extended by an $f$-invariant lamination is not necessarily invariant under $f$ nor tangent to the joint bundle $E^s\oplus E^u$.

Let us introduce some results that will be used in the subsequent.
\begin{prop}\cite{FP_hyperbolic}\label{lift}
	Let $f: M\rightarrow M$ be a homeomorphism of a compact manifold and $g:\hat{M}\rightarrow \hat{M}$ be a lift of an iterate of $f$ to a finite cover $\hat{M}$. If $NW(f)=M$, then $NW(g)=\hat{M}$. If moreover $f$ is a partially hyperbolic diffeomorphism of a closed 3-manifold, then it is accessible provided that $g$ is accessible.
\end{prop}

\begin{thm}\cite{2011TORI}\label{Anosovtorus}
	Let $f: M\rightarrow M$ be a partially hyperbolic diffeomorphism of a closed 3-manifold. Assume that there is a 2-dimensional embedded torus tangent to $E^s\oplus E^u$. Then the ambient manifold $M$ has solvable fundamental group.
\end{thm}

Recall that a codimenion-one foliation is \emph{$\R$-covered} if the leaf space of its lifted foliation in the universal cover is homeomorphic to the real numbers. For a proper lamination, the leaf space of the lifted lamination, however, is not simply connected. Alternatively, we say that a codimension-one lamination is \emph{$\R$-covered} if the leaf space of the lifted lamination is a Hausdorff separable one-dimensional manifold.

\begin{prop}\cite[Corollary 3.13]{FU1}\label{R-covered}
	Let $f: M\rightarrow M$ be a partially hyperbolic diffeomorphism of a closed 3-manifold with complete foliations $\mF^{cs}$ and $\mF^{cu}$ tangent to $E^s\oplus E^c$ and $E^c\oplus E^u$, respectively. If $\Lambda^{su}$ is a minimal lamination tangent to $E^s\oplus E^u$ without compact leaves, then each leaf of the lifted foliation $\wt{\mF}^c$ on the universal cover $\wt{M}$ intersects every leaf of $\wt{\Lambda}^{su}$ in exactly one point. Moreover, the lamination $\Lambda^{su}$ is $\R$-covered.
\end{prop}

The following proposition only concerns the manifold and a lamination, regardless of the dynamics. We recall that any closed 3-manifold admitting a partially hyperbolic system is irreducible \cite{2011TORI}.
\begin{prop}\cite[Proposition 5.7]{FP_hyperbolic}\label{Gromov}
	Let $M$ be an irreducible closed 3-manifold whose fundamental group $\pi_1(M)$ is not virtually solvable. Assume that $\Lambda\subset M$ is a minimal lamination without compact leaves such that each closed complementary region is an $I$-bundle. Then the leaves of $\Lambda$ are uniformly Gromov hyperbolic.
\end{prop}

Now, let us state our proof. 

\begin{proof}[Proof of \thmref{accessible-complete}]
	Up to taking a finite lift and iterate, we can assume that the manifold $M$ and invariant bundles $E^s$, $E^c$, and $E^u$ are all orientable and $f$ preserves these orientations. We do not lose any generality as shown in \propref{lift}. 
	
	Suppose that $f$ is a non-accessible partially hyperbolic diffeomorphism with $NW(f)=M$. If there exists a compact accessibility class, then it is an incompressible torus tangent to $E^s\oplus E^u$ since it is foliated by one-dimensional unstable manifolds. It implies by \thmref{Anosovtorus} that the fundamental group $\pi_1(M)$ has to be solvable, which contradicts to our assumption. Then, by \thmref{lamination}, the set of non-open accessibility classes $\Gamma(f)$ is either a foliation or a lamination that trivially extends to a foliation without compact leaves. We denote by $\Lambda^{su}$ a minimal sublamination of $\Gamma(f)$, whose complementary regions are $I$-bundles as shown in \thmref{lamination}.
	
	As $f$ admits two invariant foliations $\mF^{cs}$ and $\mF^{cu}$, there exists an invariant center foliation, denoted by $\mF^c$. Denote by $\wt{\mF}^{\sigma}, \sigma=cs, cu, c,$ the lifted foliations and $\wt{\Lambda}^{su}$ the lifted lamination in the universal cover $\wt{M}$.
	
	For simplicity, we can reduce the lamination $\Lambda^{su}$ to a minimal foliation by collapsing the complementary regions along center $I$-bundles. The collapsing map does not change the foliation $\mF^c$ and the lifted one $\wt{\mF}^c$ by viewing them as sets of leaves, see \propref{R-covered}. Moreover, it will not change the differentiability of the manifold and leaves of $\Lambda^{su}$ as $M$ is a 3-manifold. Then $\Lambda^{su}$ is a minimal foliation of a closed 3-manifold $M_0=M/\sim$, where two points in $M$ are equivalent if there are contained in the same $I$-fiber in a complementary region of $\Lambda^{su}$. As shown in \propref{Gromov}, the leaves of $\Lambda^{su}$ in $M$ are all uniformly Gromov hyperbolic. Then $\Lambda^{su}$ is a minimal $\R$-covered foliation in $M_0$ by non-compact uniformly Gromov hyperbolic leaves. 
	
	In the universal cover $\wt{M_0}$, each lifted leaf of $\wt{\Lambda}^{su}$ can be identified to the Poincar\'{e} disk with an ideal boundary circle at infinity. Denote by $\partial_{\infty}L$ the ideal boundary of a leaf $L\in\wt{\Lambda}^{su}$, and $\mA\colon=\bigcup_{L\in\wt{\Lambda}^{su}}\partial_{\infty}L$ the union of all ideal circles. By considering a natural topology of $\mA$, the union $\wt{M_0}\cup \mA$ is homeomorphic to a solid cylinder, see \cite[Section 4]{FU1} for details. With respect to the given topology, we can discuss about ideal points in each $\partial_{\infty}L$. Note that every deck transformation fixing some leaf of $\wt{\Lambda}^{su}$ is a hyperbolic type M\"{o}bius tranformation with exactly two fixed points in the ideal boundary of the leaf.
	
	In each leaf $L\in \wt{\Lambda}^{su}$, we can consider the subfoliation of stable leaves, denoted by $\wt{\Lambda}^s_L$, induced by the intersection with foliation $\wt{\mF}^{cs}$. As shown in \cite[Proposition 4.3]{FU1}, for any leaf $L\in\wt{\Lambda}^{su}$, each ray of any leaf of $\wt{\Lambda}^s_L$ accumulates in a single ideal point in the ideal boundary $\partial_{\infty}L$. Denoted by $S_L\subset\partial_{\infty}L$ the set of ideal limit points of $\wt{\Lambda}^s_L$.
	
	We first assume that the set $S_L$ is not dense in $\partial_{\infty}L$ with respect to the given topology for some $L\in\wt{\Lambda}^{su}$. Then, applying the same argument as in \cite[Proposition 4.16]{FU1}, the foliation $\Lambda^{su}$ is not uniform, which means that there is at least one pair of leaves in $\wt{\Lambda}^{su}$ with unbounded Hausdorff distance. By \cite[Proposition 4.11]{FU1}, $\Lambda^{su}$ coincides with the stable foliation of a flow that is conjugate to a suspension Anosov flow. Moreover, the manifold $M_0$ is a torus bundle over the circle and thus has solvable fundamental group.
	
	So, for any leaf $L\in \wt{\Lambda}^{su}$, the set $S_L$ is dense in the ideal boundary $\partial_{\infty}L$. Following the argument in \cite[Section 5]{FU1}, the completeness of $\mF^{cs}$ and $\mF^{cu}$ implies that every leaf of $\wt{\Lambda}^s_L$ has two distinct ideal points in $\partial_{\infty}L$ for any $L\in \wt{\Lambda}^{su}$. Equivalently, it turns out that all leaves of $\wt{\Lambda}^s_L$ are uniform quasi-geodesics and share a common ideal point in $\partial_{\infty}L$. Since each deck transformation preserves the stable foliation $\wt{\Lambda}^s$, the fundamental group $\pi_1(L_0)$ has at most one generator for each $L\in\wt{\Lambda}^{su}$, where $L_0=\pi(L)$ is the projection leaf of $L$ under the universal covering map $\pi: \wt{M_0}\rightarrow M_0$. By \thmref{rosenberg}, there exists at least one leaf $L\in \wt{\Lambda}^{su}$ fixing by a non-trivial deck transformation, denoted by $h(L)=L$. Let $\xi_L\in S_L$ be the common ideal point of $\wt{\Lambda}^s_L$. As $h$ is a hyperbolic M\"{o}bius transformation preserving the foliation $\wt{\Lambda}^s_L$, it fixes the point $\xi_L$ and another distinct point in $\partial_{\infty}L$, denoted by $\eta_L$. Let $\gamma_h$ be the axis of $h$, which is the unique geodesic connecting the ideal point $\xi_L$ and $\eta_L$. Since $S_L$ is dense in $\partial_{\infty}L$, by the continuity of ideal points of $\wt{\Lambda}^s_L$ (see \cite[Lemma 5.1]{FU1}), there is a leaf $l\in \wt{\Lambda}^s_L$ connecting two ideal points $\xi_L$ and $\eta_L$. Recall that $l$ is a quasi-geodesic in $L$, which means that $l$ is entirely contained in $K$-neighborhood of $\gamma_h$ for some constant $K>0$. By the fact that $\xi_L$ and $\eta_L$ are two fixed points of $h$, the orbit $\{h^i(l)\}_{i\in\Z}$ is also contained in the $K$-neighborhood of $\gamma_h$. It turns out the existence of an $h$-invariant leaf of $\wt{\Lambda}^s_L$ in the closure of the $K$-neighborhood of $\gamma_h$. Therefore, since $h$ is a deck transformation in the universal cover, there is a closed stable manifold in the projection leaf $\pi(L)\in\Lambda^{su}$, which is absurd.
	
	Hence, we finish the proof.
\end{proof}

\subsection{Proof of \thmref{accessible}}
Now we are going to present our proof of \thmref{accessible}.

\begin{proof}[Proof of \thmref{accessible}]
	Let $f: M^3\rightarrow M^3$ be a partially hyperbolic diffeomorphism with quasi-isometric center and $M$ be a closed 3-manifold whose fundamental group is not virtually solvable. By \propref{integrable} and \propref{complete}, there are $f$-invariant complete foliations $\mF^{cs}$ and $\mF^{cu}$.
	Assume that the non-wandering set of $f$ is the whole manifold, i.e., $NW(f)=M$. Then, $f$ is accessible by \thmref{accessible-complete}.
	
	If $f$ is a $C^r$ conservative partially hyperbolic diffeomorphism for $r>1$, then it is a K-system and, in particular, it is ergodic by \cite{08invent, BW10annals}. Note that accessibility is an $C^1$ open property for partially hyperbolic diffeomorphisms with one-dimensional center \cite{Didier}. There is a $C^1$-neighborhood $\mU$ of $f$ such that every diffeomorphism $g\in \mU$ is an accessible partially hyperbolic diffeomorphism. As shown in \cite{Avila10}, a volume-preserving diffeomorphism can be $C^1$-approximated by smooth volume-preserving diffeomorphisms. In particular, the set of $C^2$ volume-preserving partially hyperbolic diffeomorphisms, denoted by $\mU'$, has non-empty intersection with $\mU$. Again using \cite{08invent, BW10annals}, every diffeomorphism $g\in \mU'\cap \mU$ is a K-system and thus ergodic. Therefore, we conclude that $f$ is a stably K-system and it is stably ergodic, completing the proof.
\end{proof}

\subsection{Proof of \thmref{ergodic}}
Now, let us present our proof of \thmref{ergodic}.

\begin{proof}[Proof of \thmref{ergodic}]
	Let $f: M^3\rightarrow M^3$ be a $C^2$ partially hyperbolic diffeomorphism of a closed 3-manifold such that $f$ has quasi-isometric center and preserves the Lebesgue measure $m$. Assume that $f$ is ergodic for $m$. For any open set $U\subset M$, the union $\overline{\bigcup_{n\in \Z} f^n(U)}$ is an $f$-invariant set of positive measure. Then, its complement is a closed null measure set. Since $\overline{\bigcup_{n\in \Z} f^n(U)}$ is also a closed set, we deduce that it coincides with the manifold $M$. It implies that $f$ is transitive.
	
	Suppose that $f$ is transitive. If the fundamental group $\pi_1(M)$ is not virtually solvable, then $f$ is accessible and ergodic by \thmref{accessible}. If $\pi_1(M)$ is virtually solvable, then by \cite{HS21DA}, either $f$ is accessible or there exists an $su$-torus. We are sufficient to prove ergodicity in the later case. The set of all $su$-tori is compact and $f$-invariant by \cite{Hae62}. It implies by transitivity that $M$ is foliated by $su$-tori.
	By \propref{periodiccompact}, there is an $f$-invariant compact center leaf, denoted by $\gamma$. Since $f$ is $m$-preserving and $\|Df^n|_{E^c}\|$ is uniformly bounded, $f$ is 2-normally hyperbolic at $\gamma$ and thus $\gamma$ is a $C^2$ curve by \cite[Theorem 4.1]{HPS77}. By transitivity, the map $f|_{\gamma}$ has irrational rotation number. Applying \cite[Theorem 1.4]{Herman79}, we obtain that $f|_{\gamma}$ is ergodic. Then, for any measurable $f$-invariant $su$-saturated set $N\subset M$, the intersection $N\cap \gamma$ has either full or null $m|_{\gamma}$-measure. It follows that $N$ has either full or null $m$-measure. As shown in \cite{08invent, BW10annals}, we conclude that $f$ is ergodic and thus finish the proof.
\end{proof}

%% file: classify_NW.tex
% !Mode:: "TeX:UTF-8"

\section{Classification under non-wandering property}\label{section-classify-NW}

In this section, we establish the classification result for partially hyperbolic diffeomorphisms with quasi-isometric center and the non-wandering property, as stated in \thmref{classify_NW_strong}. Before doing so, we will present a coarser classification that allows for an additional possibility within the same context as \thmref{classify_NW_strong}. Furthermore, we will explore the relationships between these two classification results under the non-wandering condition, the classification under the transitivity assumption (\thmref{classify_transitive}), and the accessibility result (\thmref{accessible}) for partially hyperbolic diffeomorphisms with quasi-isometric center.

\subsection{A coarser classification in the non-wandering context}
We shall recall the notion of collapsed Anosov flow introduced in \cite{BFP23collapsed}, which is a wide class of partially hyperbolic diffeomorphisms in dimension three and includes all known examples of partially hyperbolic diffeomorphisms of manifolds with non-virtually solvable fundamental group.

We say that a partially hyperbolic diffeomorphism $f: M\rightarrow M$ of a clsoed 3-manifold $M$ is a \emph{collapsed Anosov flow} if there is a topological Anosov flow $\phi_t$, a self-orbit equivalence $\beta: M\rightarrow M$ of $\phi_t$ and a continuous map $h: M\rightarrow M$ homotopic to the identity such that
\begin{itemize}
	\item the map $h$ is $C^1$ along the orbits of $\phi_t$ and maps the vector field generating $\phi_t$ to a set vectors in the bundle $E^c$, that is, $\partial_th(\phi_t(x))|_{t=0}\in E^c(h(x))$ for any $x\in M$;
	\item we have that $h\circ \beta(x)=f\circ h(x)$ for any $x\in M$.
\end{itemize}

We note that a discretized Anosov flow is a particular case of collapsed Anosov flow with trivial self-orbit equivalence. One can discribe some topological center behaviors of $f$ by sending the flow lines of $\phi_t$ to center curves tangent to $E^c$ through $h$. However, a priori, the center curves given by $h$ might not be the leaves of a center foliation, or more generally, these curves might not belong to the intersection of branching foliations tangent to $E^s\oplus E^c$ and $E^c\oplus E^u$.

While the collapsed Anosov flow gathers some topological information of the center for a partially hyperbolic diffeomorphism, a more restrictive notion, called quasi-geodesic collapsed Anosov flow, provides some geometrical properties.

Given a lifted foliation $\wt{\mF}$ in the universal cover, we say that a curve $l$ in a leaf $L\in \wt{\mF}$ is a \emph{quasi-geodesic} if there are constants $\lambda>1$, $\eta>0$ and a parametrization $\chi: \R\rightarrow L$ such that $\lambda^{-1}|t-s|-\eta\leq d_l(\chi(t), \chi(s))\leq \lambda|t-s|+\eta$, where $d_l(\chi(t),\chi(s))$ is the length of subcurve joining $\chi(t)$ and $\chi(s)$ with respect to the induced metric on the leaf $L$. A family of curves is \emph{uniformly quasi-geodesic} if the constants $\lambda$ and $\eta$ can be chosen independent of curves and leaves of $\wt{\mF}$.

We say that a partially hyperbolic diffeomorphism $f: M\rightarrow M$ is a \emph{quasi-geodesic collapsed Anosov flow} if it is a collapsed Anosov flow preserving two center stable and center unstable foliations with Gromov hyperbolic leaves and moreover the center leaves are uniform quasi-geodesics inside center stable and center unstable leaves.

For the sake of simplicity, we assume the existence of invariant center stable and center unstable foliations in the definition above. Note that this definition is strictly more restrictive than the one introduced in \cite{BFP23collapsed}, where the authors consider invariant branching foliations in the sense of \cite{BI08}.

The following result provides a weaker version of classification compared with \thmref{classify_NW_strong}, which is independent of \thmref{classify_transitive}. As we mentioned before, this result is also weaker than the one independently obtained in \cite{EMP}. As the idea of proof here is analogous to \cite{EMP}, we will only give a sketch without much detail.

\begin{thm}\label{classify_NW}
	Let $f: M^3\rightarrow M^3$ be a $C^1$ partially hyperbolic diffeomorphism of a closed 3-manifold with quasi-isometric center and $NW(f)=M$. Then, up to a finite lift and iterate, one of the following occurs:
	\begin{itemize}
		\item $f$ is conjugate to a skew product over an Anosov automorphism on $\T^2$;
		\item $f$ is a discretized Anosov flow;
		\item $f$ is a quasi-geodesic collapsed Anosov flow.
	\end{itemize}
\end{thm}

\begin{proof}
	Notice that the partially hyperbolic diffeomorphism $f$ with quasi-isometric center is dynamically coherent. We devide the proof into two cases: either the fundamental group of the manifold $\pi_1(M)$ is virtually solvable or not. If $\pi_1(M)$ is virtually solvable, then the diffeomorphism $f$ admits an iterate over a finite lift which either is conjugate to a skew product over an Anosov automorphism on $\T^2$ or is a discretized suspension flow over an Anosov map \cite{HP15}. So, we are sufficient to discuss the situation that $\pi_1(M)$ is not virtually solvable.
	
	Note that, by the condition $NW(f)=M$, both invariant foliations $\mF^{cs}$ and $\mF^{cu}$ are minimal \cite[Proposition 6.7]{2020Seifert}, which means that every leaf is dense in $M$. Since the 3-manifold $M$ is irreducible as shown in \cite{2011TORI}, the leaves of $\mF^{cs}$ and $\mF^{cu}$ are all Gromov hyperbolic by \propref{Gromov}. By the completeness given by \propref{complete}, the lifted foliations $\wt{\mF}^{cs}$ and $\wt{\mF}^{cu}$ in the universal cover $\wt{M}$ are also complete. It implies that, in any leaf of $\wt{\mF}^{cs}$, every center leaf intersects each stable leaf in a unique point. To see this, we first notice that any center leaf must intersect every stable leaf inside the same leaf of $\wt{\mF}^{cs}$ by the completeness. Suppose that the intersection of a pair of center and stable leaves contains at least two points. Note that the leaves of $\wt{\mF}^{cs}$ are properly embedded planes in $\R^3$. Then, the Poincar\'{e}-Bendixson Theorem implies that either there is a compact stable leaf, or there is a point with degenerated stable tangent space. Both of these two possibilities are absurd by the partial hyperbolicity. 
	
	Thus, in each leaf of $\wt{\mF}^{cs}$, the leaf space of the center foliation is homeomorphic to a stable leaf, and thus it is homeomorphic to the real numbers. The same holds for $\wt{\mF}^{cu}$. Applying \cite[Theorem 1.1]{FP23intersection}, the leaves of $\wt{\mF}^c$ are uniformly quasi-geodesic. Recall in \corref{cylinder/plane} that each leaf of $\mF^{cs}$ and $\mF^{cu}$ is either a cylinder or a plane. By considering a $\pi_1(M)$-equivariant vector field tangent to the leaves of $\wt{\mF}^c$, one can produce a non-singular flow $\phi_t$ preserving both foliations $\mF^{cs}$ and $\mF^{cu}$ projected by the covering map $\pi: \wt{M}\rightarrow M$. Moreover, the flow $\phi_t$ is expansive by the fact that the leaves of $\wt{\mF}^c$ are uniform quasi-geodesics with a single common ideal point inside each leaf of $\wt{\mF}^{cs}$ and $\wt{\mF}^{cu}$, see \cite[Proposition 7.1]{BFP23collapsed}. Thus, the flow $\phi_t$ is a topological Anosov flow, see for instance \cite{IM90, Paternain93} or \cite[Theorem 5.9]{BFP23collapsed}. It follows directly from the construction of $\phi_t$ that $f$ is a quasi-geodesic collapsed Anosov flow. Hence, the proof is finished.
\end{proof}

\subsection{An alternative proof of accessibility}

We present an alternative proof of \thmref{accessible} by making use of our classification, which is different from the proof we provided in \sref{section-accessible}.

By \propref{lift}, it suffices to show the accessibility for any finite lift and iterate of a partially hyperbolic diffeomorphism with quasi-isometric center. As shown in \thmref{classify_NW}, up to a finite lift and iterate, we can assume that either $f$ is conjugate to a skew product over an Anosov automorphism on $\T^2$, it is a discretized Anosov flow, or it is a quasi-geodesic collapsed Anosov flow. 

In the first case, the manifold $M$ is a circle bundle over a torus and thus its fundamental group is solvable. If $f$ is a discretized Anosov flow, then it is accessible by \cite[Theorem C]{FP_hyperbolic} since $\pi_1(M)$ is not virtually solvable. The accessibility in the case that $f$ is a quasi-geodesic collapsed Anosov flow can be obtained directly using the following result.

\begin{thm}\cite{FP21accessible}
	Let $f: M\rightarrow M$ be a collapsed Anosov flow of a closed 3-manifold with $NW(f)=M$. Assume that $\pi_1(M)$ is not virtually solvable. Then, $f$ is accessible.
\end{thm}

Therefore, we complete the proof of \thmref{accessible}.

\subsection{Proof of \thmref{classify_NW_strong}}

Let $f: M\rightarrow M$ be a $C^1$ partially hyperbolic diffeomorphism of a closed 3-manifold with quasi-isometric center. We have shown in \propref{integrable} that the center bundle $E^c$ is uniquely integrable. In particular, $f$ is dynamically coherent. 

If the fundamental group $\pi_1(M)$ is virtually solvable, then by \cite{HP15} the dynamical coherence implies that, up to a finite lift and iterate, $f$ is either a derived map from Anosov, conjugate to a skew product over an Anosov automorphism on $\T^2$, or a discretized Anosov flow. The first case cannot occur since it is homotopic to a uniformly hyperbolic map and thus cannot have quasi-isometric center. In the last case, we assume that $\hat{f^k}: \hat{M}\rightarrow \hat{M}$ is a lift of some iterate $f^k: M\rightarrow M$ by a finite covering $\pi: \hat{M}\rightarrow M$ such that $\hat{f^k}$ is a discretized Anosov flow. Then $\hat{f^k}$ fixes every leaf of the lifted foliation $\hat{\mF^c}$. For any $\hat{f^k}$-invariant compact center leaf $\hat{\gamma}\in\hat{\mF^c}$, its projection $\pi(\hat{\gamma})=:\gamma$ is a compact $f$-periodic center leaf of $\mF^c$. Moreover, every center leaf in $W^s(\gamma)$ and $W^u(\gamma)$ is $f$-periodic since it is a projection of an $\hat{f^k}$-invariant center leaf. Thus, $f$ satisfies the assumption of second item of \thmref{BW} and then it admits an iterate that is a discretized Anosov flow.

When $\pi_1(M)$ is not virtually solvable, the diffeomorphism $f$ is accessible given by \thmref{accessible}. In light of \cite{Brin75_transitive}, the non-wandering property is equivalent to the transitivity for accessible partially hyperbolic diffeomorphism. Then, we can apply the result of \thmref{classify_transitive} to get the complete classification. Hence, we finish the proof.

\subsection{Proof of \thmref{flow}}

Let $M$ be a closed 3-manifold whose fundamental group $\pi_1(M)$ has exponential growth, and $f: M\rightarrow M$ be a partially hyperbolic diffeomorphism with quasi-isometric center and $NW(f)=M$. By \thmref{classify_NW_strong}, if any iterate of $f$ is not a discretized Anosov flow, then there is a lift of an iterate of $f$ conjugate to a skew product over an Anosov automorphism on the torus. It turns out that $M$ is finitely covered by a circle bundle over the torus, and thus its fundamental group $\pi_1(M)$ does not have exponential growth.

Assume that $f^k: M\rightarrow M$ is a discretized Anosov flow for some integer $k\in \N$. We deduce from \thmref{BW} that the center foliation of $f^k$ coincides with a transitive topological Anosov flow, denoted by $\phi_t: M\rightarrow M$. By \cite{Shannon_flow}, there exists a smooth Anosov flow $\psi_t: M\rightarrow M$ orbit equivalent to $\phi_t$. In particular, $\psi_t$ is also transitive. Hence, we complete the proof.

%% file: QI.bbl
\providecommand{\bysame}{\leavevmode\hbox to3em{\hrulefill}\thinspace}
\providecommand{\noopsort}[1]{}
\providecommand{\mr}[1]{\href{http://www.ams.org/mathscinet-getitem?mr=#1}{MR~#1}}
\providecommand{\zbl}[1]{\href{http://www.zentralblatt-math.org/zmath/en/search/?q=an:#1}{Zbl~#1}}
\providecommand{\jfm}[1]{\href{http://www.emis.de/cgi-bin/JFM-item?#1}{JFM~#1}}
\providecommand{\arxiv}[1]{\href{http://www.arxiv.org/abs/#1}{arXiv~#1}}
\providecommand{\doi}[1]{\url{https://doi.org/#1}}
\providecommand{\MR}{\relax\ifhmode\unskip\space\fi MR }
% \MRhref is called by the amsart/book/proc definition of \MR.
\providecommand{\MRhref}[2]{%
  \href{http://www.ams.org/mathscinet-getitem?mr=#1}{#2}
}
\providecommand{\href}[2]{#2}
\begin{thebibliography}{CHHU18}

\bibitem[Ano67]{Anosov1967}
\bgroup\scshape{}D.~V. Anosov\egroup{}, Geodesic flows on closed riemannian manifolds of negative curvature,  \emph{Trudy Matematicheskogo Instituta Imeni VA Steklova} \textbf{90} (1967), 3--210.

\bibitem[AS67]{AnosovSinai67}
\bgroup\scshape{}D.~V. Anosov\egroup{} and \bgroup\scshape{}J.~G. Sinai\egroup{}, Certain smooth ergodic systems,  \emph{Uspehi Mat. Nauk} \textbf{22} no.~5(137) (1967), 107--172.

\bibitem[ACW21]{ACW21}
\bgroup\scshape{}A.~Avila\egroup{}, \bgroup\scshape{}S.~Crovisier\egroup{}, and \bgroup\scshape{}A.~Wilkinson\egroup{}, C1 density of stable ergodicity,  \emph{Advances in Mathematics} \textbf{379} (2021), 107496.

\bibitem[Avi10]{Avila10}
\bgroup\scshape{}A.~Avila\egroup{}, On the regularization of conservative maps,  \emph{Acta mathematica} \textbf{205} no.~1 (2010), 5--18.

\bibitem[Bar98]{Barbot98}
\bgroup\scshape{}T.~Barbot\egroup{}, Generalizations of the bonatti--langevin example of anosov flow and their classification up to topological equivalence,  \emph{Communications in Analysis and Geometry} \textbf{6} no.~4 (1998), 749--798.

\bibitem[BF13]{BarbotFenley13}
\bgroup\scshape{}T.~Barbot\egroup{} and \bgroup\scshape{}S.~R. Fenley\egroup{}, Pseudo-anosov flows in toroidal manifolds,  \emph{Geometry \& Topology} \textbf{17} no.~4 (2013), 1877--1954.

\bibitem[BFM23]{BFM23}
\bgroup\scshape{}T.~Barthelm{\'e}\egroup{}, \bgroup\scshape{}S.~Fenley\egroup{}, and \bgroup\scshape{}K.~Mann\egroup{}, Anosov flows with the same periodic orbits,  \emph{arXiv preprint arXiv:2308.02098} (2023).

\bibitem[BFFP19]{BFFP1}
\bgroup\scshape{}T.~Barthelm{\'e}\egroup{}, \bgroup\scshape{}S.~R. Fenley\egroup{}, \bgroup\scshape{}S.~Frankel\egroup{}, and \bgroup\scshape{}R.~Potrie\egroup{}, Partially hyperbolic diffeomorphisms homotopic to the identity in dimension 3, part i: The dynamically coherent case,  \emph{arXiv preprint arXiv:1908.06227} (2019).

\bibitem[BFFP23]{BFFP2}
\bgroup\scshape{}T.~Barthelm{\'e}\egroup{}, \bgroup\scshape{}S.~R. Fenley\egroup{}, \bgroup\scshape{}S.~Frankel\egroup{}, and \bgroup\scshape{}R.~Potrie\egroup{}, Partially hyperbolic diffeomorphisms homotopic to the identity in dimension 3, ii: Branching foliations,  \emph{Geometry \& Topology} \textbf{27} no.~8 (2023), 3095--3181.

\bibitem[BFP23]{BFP23collapsed}
\bgroup\scshape{}T.~Barthelm{\'e}\egroup{}, \bgroup\scshape{}S.~R. Fenley\egroup{}, and \bgroup\scshape{}R.~Potrie\egroup{}, Collapsed anosov flows and self orbit equivalences,  \emph{Commentarii Mathematici Helvetici} \textbf{98} no.~4 (2023), 771--875.

\bibitem[BFM22]{BFM22transitive}
\bgroup\scshape{}T.~Barthelm{\'e}\egroup{}, \bgroup\scshape{}S.~Frankel\egroup{}, and \bgroup\scshape{}K.~Mann\egroup{}, Orbit equivalences of pseudo-anosov flows,  \emph{arXiv preprint arXiv:2211.10505} (2022).

\bibitem[BBM24]{BBM24nontransitive}
\bgroup\scshape{}T.~Barthelm{\'e}\egroup{}, \bgroup\scshape{}C.~Bonatti\egroup{}, and \bgroup\scshape{}K.~Mann\egroup{}, Non-transitive pseudo-anosov flows,  \emph{arXiv preprint arXiv:2411.03586} (2024).

\bibitem[BGHP20]{BGHP3}
\bgroup\scshape{}C.~Bonatti\egroup{}, \bgroup\scshape{}A.~Gogolev\egroup{}, \bgroup\scshape{}A.~Hammerlindl\egroup{}, and \bgroup\scshape{}R.~Potrie\egroup{}, Anomalous partially hyperbolic diffeomorphisms iii: abundance and incoherence,  \emph{Geometry \& Topology} \textbf{24} no.~4 (2020), 1751--1790.

\bibitem[BGP16]{BGP2}
\bgroup\scshape{}C.~Bonatti\egroup{}, \bgroup\scshape{}A.~Gogolev\egroup{}, and \bgroup\scshape{}R.~Potrie\egroup{}, Anomalous partially hyperbolic diffeomorphisms ii: stably ergodic examples,  \emph{Inventiones Mathematicae} \textbf{206} no.~3 (2016), 801--836.

\bibitem[BL98]{BonattiLangevin98}
\bgroup\scshape{}C.~Bonatti\egroup{} and \bgroup\scshape{}R.~Langevin\egroup{}, Diff\'eomorphismes de {S}male des surfaces,  \emph{Ast\'erisque} no.~250 (1998), viii+235, With the collaboration of E. Jeandenans.

\bibitem[BPP16]{BPP1}
\bgroup\scshape{}C.~Bonatti\egroup{}, \bgroup\scshape{}K.~Parwani\egroup{}, and \bgroup\scshape{}R.~Potrie\egroup{}, Anomalous partially hyperbolic diffeomorphisms i: dynamically coherent examples,  \emph{Annales Scientifiques de l'{\'E}cole Normale Sup{\'e}rieure} \textbf{49} no.~6 (2016), 1387--1402.

\bibitem[BW05]{BW05}
\bgroup\scshape{}C.~Bonatti\egroup{} and \bgroup\scshape{}A.~Wilkinson\egroup{}, Transitive partially hyperbolic diffeomorphisms on 3-manifolds,  \emph{Topology} \textbf{44} no.~3 (2005), 475--508.

\bibitem[BZ20]{BZ20}
\bgroup\scshape{}C.~Bonatti\egroup{} and \bgroup\scshape{}J.~Zhang\egroup{}, Transitive partially hyperbolic diffeomorphisms with one-dimensional neutral center,  \emph{Science China Mathematics} \textbf{63} no.~9 (2020), 1647--1670.

\bibitem[Bow75]{Bowen_book}
\bgroup\scshape{}R.~Bowen\egroup{}, Equilibrium states and the ergodic theory of anosov diffeomorphisms,  \emph{Lecture notes in mathematics} \textbf{470} (1975), 11--25.

\bibitem[Bri75]{Brin75_transitive}
\bgroup\scshape{}M.~Brin\egroup{}, Topological transitivity of one class of dynamic systems and flows of frames on manifolds of negative curvature,  \emph{Functional Analysis and Its Applications} \textbf{9} no.~1 (1975), 8--16.

\bibitem[BP74]{BrinPesin74}
\bgroup\scshape{}M.~I. Brin\egroup{} and \bgroup\scshape{}J.~B. Pesin\egroup{}, Partially hyperbolic dynamical systems,  \emph{Izv. Akad. Nauk SSSR Ser. Mat.} \textbf{38} (1974), 170--212.

\bibitem[BI08]{BI08}
\bgroup\scshape{}D.~Burago\egroup{} and \bgroup\scshape{}S.~Ivanov\egroup{}, Partially hyperbolic diffeomorphisms of 3-manifolds with abelian fundamental groups,  \emph{Journal of Modern Dynamics} \textbf{2} no.~4 (2008), 541.

\bibitem[BW08]{BurnsWilkinson08}
\bgroup\scshape{}K.~Burns\egroup{} and \bgroup\scshape{}A.~Wilkinson\egroup{}, Dynamical coherence and center bunching,  \emph{Discrete Contin. Dyn. Syst} \textbf{22} no.~1-2 (2008), 89--100.

\bibitem[BW10]{BW10annals}
\bgroup\scshape{}K.~Burns\egroup{} and \bgroup\scshape{}A.~Wilkinson\egroup{}, On the ergodicity of partially hyperbolic systems,  \emph{Annals of Mathematics} (2010), 451--489.

\bibitem[CD03]{CalegariDunfield03}
\bgroup\scshape{}D.~Calegari\egroup{} and \bgroup\scshape{}N.~Dunfield\egroup{}, Laminations and groups of homeomorphisms of the circle,  \emph{Inventiones mathematicae} \textbf{152} no.~1 (2003), 149--204.

\bibitem[CHHU18]{2018survey}
\bgroup\scshape{}P.~Carrasco\egroup{}, \bgroup\scshape{}F.~R. Hertz\egroup{}, \bgroup\scshape{}J.~R. Hertz\egroup{}, and \bgroup\scshape{}R.~Ures\egroup{}, Partially hyperbolic dynamics in dimension three,  \emph{Ergodic Theory and Dynamical Systems} \textbf{38} no.~8 (2018), 2801--2837.

\bibitem[Car15]{Carrasco15compact}
\bgroup\scshape{}P.~D. Carrasco\egroup{}, Compact dynamical foliations,  \emph{Ergodic theory and dynamical systems} \textbf{35} no.~8 (2015), 2474--2498.

\bibitem[CPH21]{CPH21}
\bgroup\scshape{}P.~D. Carrasco\egroup{}, \bgroup\scshape{}E.~Pujals\egroup{}, and \bgroup\scshape{}F.~R. Hertz\egroup{}, Classification of partially hyperbolic diffeomorphisms under some rigid conditions,  \emph{Ergodic Theory and Dynamical Systems} \textbf{41} no.~9 (2021), 2770--2781.

\bibitem[CP22]{CP22_principle}
\bgroup\scshape{}S.~Crovisier\egroup{} and \bgroup\scshape{}M.~Poletti\egroup{}, Invariance principle and non-compact center foliations,  \emph{arXiv preprint arXiv:2210.14989} (2022).

\bibitem[DMM20]{DeMartinchich20}
\bgroup\scshape{}V.~De~Martino\egroup{} and \bgroup\scshape{}S.~Martinchich\egroup{}, Codimension one compact center foliations are uniformly compact,  \emph{Ergodic Theory and Dynamical Systems} \textbf{40} no.~9 (2020), 2349--2367.

\bibitem[DPU99]{DPU}
\bgroup\scshape{}L.~J. D{\'\i}az\egroup{}, \bgroup\scshape{}E.~R. Pujals\egroup{}, and \bgroup\scshape{}R.~Ures\egroup{}, Partial hyperbolicity and robust transitivity,  \emph{Acta Mathematica} \textbf{183} no.~1 (1999), 1--43.

\bibitem[Did03]{Didier}
\bgroup\scshape{}P.~Didier\egroup{}, Stability of accessibility,  \emph{Ergodic Theory and Dynamical Systems} \textbf{23} no.~6 (2003), 1717--1731.

\bibitem[EMP]{EMP}
\bgroup\scshape{}M.~Espitia\egroup{}, \bgroup\scshape{}S.~Martinchich\egroup{}, and \bgroup\scshape{}R.~Potrie\egroup{}, Quasi-isometric center action in dimension 3,  \emph{In preparation}.

\bibitem[FU24]{FU1}
\bgroup\scshape{}Z.~Feng\egroup{} and \bgroup\scshape{}R.~Ures\egroup{}, Accessibility and ergodicity of partially hyperbolic diffeomorphisms without periodic points,  \emph{arXiv preprint arXiv:2404.07062} (2024).

\bibitem[Fen07]{Fenley07}
\bgroup\scshape{}S.~Fenley\egroup{}, Laminar free hyperbolic 3-manifolds,  \emph{Commentarii Mathematici Helvetici} \textbf{82} no.~2 (2007), 247--321.

\bibitem[FP24]{FP-gafa}
\bgroup\scshape{}S.~Fenley\egroup{} and \bgroup\scshape{}R.~Potrie\egroup{}, Partial hyperbolicity and pseudo-anosov dynamics,  \emph{Geometric and Functional Analysis} \textbf{34} no.~2 (2024), 409--485.

\bibitem[FP21]{FP21accessible}
\bgroup\scshape{}S.~R. Fenley\egroup{} and \bgroup\scshape{}R.~Potrie\egroup{}, Accessibility and ergodicity for collapsed anosov flows,  \emph{arXiv preprint arXiv:2103.14630} (2021).

\bibitem[FP22]{FP_hyperbolic}
\bgroup\scshape{}S.~R. Fenley\egroup{} and \bgroup\scshape{}R.~Potrie\egroup{}, Ergodicity of partially hyperbolic diffeomorphisms in hyperbolic 3-manifolds,  \emph{Advances in Mathematics} \textbf{401} (2022), 108315.

\bibitem[FP23a]{FP23intersection}
\bgroup\scshape{}S.~R. Fenley\egroup{} and \bgroup\scshape{}R.~Potrie\egroup{}, Intersection of transverse foliations in 3-manifolds: Hausdorff leafspace implies leafwise quasi-geodesic,  \emph{arXiv preprint arXiv:2310.05176} (2023).

\bibitem[FP23b]{FP23_transverse}
\bgroup\scshape{}S.~R. Fenley\egroup{} and \bgroup\scshape{}R.~Potrie\egroup{}, Transverse minimal foliations on unit tangent bundles and applications,  \emph{arXiv preprint arXiv:2303.14525} (2023).

\bibitem[FW80]{FranksWilliams}
\bgroup\scshape{}J.~Franks\egroup{} and \bgroup\scshape{}B.~Williams\egroup{}, Anomalous anosov flows,  \emph{Lecture Notes in Mathematics, Berlin Springer Verlag} \textbf{819} (1980), 158.

\bibitem[Fri83]{Fried83}
\bgroup\scshape{}D.~Fried\egroup{}, Transitive anosov flows and pseudo-anosov maps,  \emph{Topology} \textbf{22} no.~3 (1983), 299--303.

\bibitem[Gab90]{Gabai90}
\bgroup\scshape{}D.~Gabai\egroup{}, Foliations and 3-manifolds,  in \emph{Proceedings of the International Congress of Mathematicians}, \textbf{1}, 1990, pp.~609--619.

\bibitem[GS20]{GS20DA}
\bgroup\scshape{}S.~Gan\egroup{} and \bgroup\scshape{}Y.~Shi\egroup{}, Rigidity of center lyapunov exponents and $ su $-integrability,  \emph{Commentarii Mathematici Helvetici} \textbf{95} no.~3 (2020), 569--592.

\bibitem[Ghy84]{Ghys84}
\bgroup\scshape{}{\'E}.~Ghys\egroup{}, Flots d'anosov sur les 3-vari{\'e}t{\'e}s fibr{\'e}es en cercles,  \emph{Ergodic Theory and Dynamical Systems} \textbf{4} no.~1 (1984), 67--80.

\bibitem[Goo83]{Goodman83}
\bgroup\scshape{}S.~Goodman\egroup{}, Dehn surgery on anosov flows,  \emph{Lecture Notes in Mathematics} (1983), 300--307.

\bibitem[Hae]{Hae62}
\bgroup\scshape{}A.~Haefliger\egroup{}, \emph{Varietes Feuilletees}.

\bibitem[Ham13]{Hammerlindl13}
\bgroup\scshape{}A.~Hammerlindl\egroup{}, Leaf conjugacies on the torus,  \emph{Ergodic theory and dynamical systems} \textbf{33} no.~3 (2013), 896--933.

\bibitem[HP14]{HP14}
\bgroup\scshape{}A.~Hammerlindl\egroup{} and \bgroup\scshape{}R.~Potrie\egroup{}, Pointwise partial hyperbolicity in three-dimensional nilmanifolds,  \emph{Journal of the London Mathematical Society} \textbf{89} no.~3 (2014), 853--875.

\bibitem[HP15]{HP15}
\bgroup\scshape{}A.~Hammerlindl\egroup{} and \bgroup\scshape{}R.~Potrie\egroup{}, Classification of partially hyperbolic diffeomorphisms in 3-manifolds with solvable fundamental group,  \emph{Journal of Topology} \textbf{8} no.~3 (2015), 842--870.

\bibitem[HP18]{HP18_survey}
\bgroup\scshape{}A.~Hammerlindl\egroup{} and \bgroup\scshape{}R.~Potrie\egroup{}, Partial hyperbolicity and classification: a survey,  \emph{Ergodic Theory and Dynamical Systems} \textbf{38} no.~2 (2018), 401--443.

\bibitem[HS21]{HS21DA}
\bgroup\scshape{}A.~Hammerlindl\egroup{} and \bgroup\scshape{}Y.~Shi\egroup{}, Accessibility of derived-from-anosov systems,  \emph{Transactions of the American Mathematical Society} \textbf{374} no.~4 (2021), 2949--2966.

\bibitem[HU14]{HamU14CCM}
\bgroup\scshape{}A.~Hammerlindl\egroup{} and \bgroup\scshape{}R.~Ures\egroup{}, Ergodicity and partial hyperbolicity on the 3-torus,  \emph{Communications in Contemporary Mathematics} \textbf{16} no.~04 (2014), 1350038.

\bibitem[HHU20]{2020Seifert}
\bgroup\scshape{}A.~Hammerlindl\egroup{}, \bgroup\scshape{}J.~R. Hertz\egroup{}, and \bgroup\scshape{}R.~Ures\egroup{}, Ergodicity and partial hyperbolicity on seifert manifolds,  \emph{Journal of Modern Dynamics} (2020), 331.

\bibitem[HT80]{HT80}
\bgroup\scshape{}M.~Handel\egroup{} and \bgroup\scshape{}W.~P. Thurston\egroup{}, Anosov flows on new three manifolds,  \emph{Inventiones mathematicae} \textbf{59} (1980), 95--103.

\bibitem[Her79]{Herman79}
\bgroup\scshape{}M.~R. Herman\egroup{}, Sur la conjugaison diff{\'e}rentiable des diff{\'e}omorphismes du cercle {\`a} des rotations,  \emph{Publications Math{\'e}matiques de l'IH{\'E}S} \textbf{49} (1979), 5--233.

\bibitem[HHU11]{2011TORI}
\bgroup\scshape{}F.~R. Hertz\egroup{}, \bgroup\scshape{}J.~R. Hertz\egroup{}, and \bgroup\scshape{}R.~Ures\egroup{}, Tori with hyperbolic dynamics in 3-manifolds,  \emph{Journal of Modern Dynamics} \textbf{5} no.~1 (2011), 185--202.

\bibitem[HHU15]{2015Center}
\bgroup\scshape{}F.~R. Hertz\egroup{}, \bgroup\scshape{}J.~R. Hertz\egroup{}, and \bgroup\scshape{}R.~Ures\egroup{}, Center-unstable foliations do not have compact leaves,  \emph{Mathematical Research Letters} (2015).

\bibitem[HHU16]{2016example}
\bgroup\scshape{}F.~R. Hertz\egroup{}, \bgroup\scshape{}J.~R. Hertz\egroup{}, and \bgroup\scshape{}R.~Ures\egroup{}, A non-dynamically coherent example on t3,  \emph{Annales de l'Institut Henri Poincaré C, Analyse non linéaire} \textbf{33} no.~4 (2016), 1023--1032.

\bibitem[HHU08a]{08invent}
\bgroup\scshape{}F.~R. Hertz\egroup{}, \bgroup\scshape{}M.~R. Hertz\egroup{}, and \bgroup\scshape{}R.~Ures\egroup{}, Accessibility and stable ergodicity for partially hyperbolic diffeomorphisms with 1d-center bundle,  \emph{Inventiones Mathematicae} \textbf{2} no.~172 (2008), 353--381.

\bibitem[HHU08b]{2008nil}
\bgroup\scshape{}F.~R. Hertz\egroup{}, \bgroup\scshape{}M.~R. Hertz\egroup{}, and \bgroup\scshape{}R.~Ures\egroup{}, Partial hyperbolicity and ergodicity in dimension three,  \emph{Journal of Modern Dynamics} \textbf{2} no.~2 (2008), 187--208.

\bibitem[HPS77]{HPS77}
\bgroup\scshape{}M.~Hirsch\egroup{}, \bgroup\scshape{}C.~Pugh\egroup{}, and \bgroup\scshape{}M.~Shub\egroup{}, Invariant manifolds. springer lecture notes in mathematics, 583.,  (1977).

\bibitem[H{\"o}l01]{Holder1901}
\bgroup\scshape{}O.~H{\"o}lder\egroup{}, \emph{Die axiome der quantit{\"a}t und die lehre vom mass}, Teubner, 1901.

\bibitem[Hop39]{Hopf1939}
\bgroup\scshape{}E.~Hopf\egroup{}, \emph{Statistik der geod{\"a}tischen Linien in Mannigfaltigkeiten negativer Kr{\"u}mmung}, Hirzel, 1939.

\bibitem[IM90]{IM90}
\bgroup\scshape{}T.~Inaba\egroup{} and \bgroup\scshape{}S.~Matsumoto\egroup{}, Nonsingular expansive flows on 3-manifolds and foliations with circle prong singularities,  \emph{Japanese journal of mathematics. New series} \textbf{16} no.~2 (1990), 329--340.

\bibitem[Mar67]{Margulis67}
\bgroup\scshape{}G.~Margulis\egroup{}, Y-flows on three-dimensional manifolds, appendix to anosov-sinai:“some smooth ergodic systems”. uspekhi math. nauk, 22 (1967), 107-172,  \emph{Russian Math. Surveys} \textbf{22} (1967), 103--168.

\bibitem[Mar23]{Martinchich23}
\bgroup\scshape{}S.~Martinchich\egroup{}, Global stability of discretized anosov flows,  \emph{Journal of Modern Dynamics} \textbf{19} (2023), 561--623.

\bibitem[Mos92a]{Mosher92-1}
\bgroup\scshape{}L.~Mosher\egroup{}, Dynamical systems and the homology norm of a 3-manifold, i: Efficient intersection of surfaces and flows,  \emph{Duke Mathematical Journal} \textbf{65} no.~3 (1992), 449--500.

\bibitem[Mos92b]{Mosher92-2}
\bgroup\scshape{}L.~Mosher\egroup{}, Dynamical systems and the homology norm of a 3-manifold ii,  \emph{Inventiones Mathematicae} \textbf{107} no.~1 (1992), 243--281.

\bibitem[Pat93]{Paternain93}
\bgroup\scshape{}M.~Paternain\egroup{}, Expansive flows and the fundamental group,  \emph{Boletim da Sociedade Brasileira de Matem{\'a}tica} \textbf{24} no.~2 (1993), 179--199.

\bibitem[Pla81]{Plante81}
\bgroup\scshape{}J.~F. Plante\egroup{}, Anosov flows, transversely affine foliations, and a conjecture of verjovsky,  \emph{Journal of the London Mathematical Society} \textbf{2} no.~2 (1981), 359--362.

\bibitem[PT72]{PT72}
\bgroup\scshape{}J.~F. Plante\egroup{} and \bgroup\scshape{}W.~P. Thurston\egroup{}, Anosov flows and the fundamental group,  \emph{Topology} \textbf{11} no.~2 (1972), 147--150.

\bibitem[Pot18]{Potrie18ICM}
\bgroup\scshape{}R.~Potrie\egroup{}, Robust dynamics, invariant structures and topological classification,  in \emph{Proceedings of the International Congress of Mathematicians 2018}, World Scientific, 2018, pp.~2063--2085.

\bibitem[PS72]{PughShub72}
\bgroup\scshape{}C.~Pugh\egroup{} and \bgroup\scshape{}M.~Shub\egroup{}, Ergodicity of {A}nosov actions,  \emph{Inventiones Mathematicae} \textbf{15} (1972), 1--23.

\bibitem[RSS03]{RSS03}
\bgroup\scshape{}R.~Roberts\egroup{}, \bgroup\scshape{}J.~Shareshian\egroup{}, and \bgroup\scshape{}M.~Stein\egroup{}, Infinitely many hyperbolic 3-manifolds which contain no reebless foliation,  \emph{Journal of the American Mathematical Society} \textbf{16} no.~3 (2003), 639--679.

\bibitem[Ros68]{Rosenberg}
\bgroup\scshape{}H.~Rosenberg\egroup{}, Foliations by planes,  \emph{Topology} \textbf{7} no.~2 (1968), 131--138.

\bibitem[Sha21]{Shannon_flow}
\bgroup\scshape{}M.~Shannon\egroup{}, Hyperbolic models for transitive topological anosov flows in dimension three,  \emph{arXiv preprint arXiv:2108.12000} (2021).

\bibitem[Thu78]{Thurston78-note}
\bgroup\scshape{}W.~P. Thurston\egroup{}, The geometry and topology of 3-manifolds,  \emph{Lecture note} (1978).

\bibitem[Yu23]{Yu23}
\bgroup\scshape{}B.~Yu\egroup{}, Anosov flows on dehn surgeries on the figure-eight knot,  \emph{Duke Mathematical Journal} \textbf{172} no.~11 (2023), 2195--2240.

\bibitem[Zha21]{Zhang21}
\bgroup\scshape{}J.~Zhang\egroup{}, Partially hyperbolic diffeomorphisms with one-dimensional neutral center on 3-manifolds,  \emph{Journal of Modern Dynamics} \textbf{17} no.~0 (2021), 557--584.

\end{thebibliography}
